\documentclass[12pt]{article}

\usepackage{CJK}
\usepackage{amsmath}
\usepackage{amssymb}
\usepackage{amsthm}
\usepackage{graphics}
\usepackage{graphicx}
\usepackage{subfigure}

\usepackage{bm}
\usepackage{geometry}
\usepackage{todonotes}

\usepackage{multirow}

\usepackage{booktabs}
\usepackage{color}
\usepackage[colorlinks,
            linkcolor=blue,
            anchorcolor=blue,
            citecolor=blue
            ]{hyperref}

\newcommand{\normmm}[1]{{\left\vert\kern-0.25ex\left\vert\kern-0.25ex\left\vert #1
    \right\vert\kern-0.25ex\right\vert\kern-0.25ex\right\vert}}

\begin{document}
\newtheorem{Proposition}{Proposition}[section]
\newtheorem{Remark}{Remark}[section]
\newtheorem{Theorem}[Proposition]{Theorem}
\newtheorem{Lemma}[Proposition]{Lemma}

\newcommand{\RNum}[1]{\uppercase\expandafter{\romannumeral #1\relax}}
\newcommand{\rNum}[1]{\romannumeral #1\relax}
\date{}

\title{Numerical schemes to reconstruct three dimensional time-dependent point sources of acoustic waves}

\author{Bo Chen$^{\rm a}$, Yukun Guo$^{\rm b}$, Fuming Ma$^{\rm c}$, Yao Sun$^{\rm a}$$^{\ast}$\thanks{$^\ast$Corresponding author. College of Science, Civil Aviation University of China, 2898 Jinbei Road, Dongli District, Tianjin 300300, China; syhf2008@gmail.com}\\
\small{$^{a}${\em{College of Science, Civil Aviation University of China, Tianjin, China}}};\\ \small{$^{b}${\em{School of Mathematics, Harbin Institute of Technology, Harbin, China}}}; \\ \small{$^{c}${\em{Institute of Mathematics, Jilin University, Changchun, China}}}}

\maketitle

\begin{abstract}
\small
This paper is concerned with the numerical simulation of three dimensional time-dependent inverse source problems of acoustic waves. The reconstructions of both multiple stationary point sources and a moving point source are considered. The modified method of fundamental solutions (MMFS), which expands the solution utilizing the time convolution of the Green's function and the signal function, is proposed to solve the problem. For the reconstruction of a moving point source, moreover, the MMFS is simplified as a simple sampling method at each time step. Numerical experiments are provided to show the effectiveness of the proposed methods.

\vspace{3mm}

\noindent\textbf{Keywords:} time-dependent, inverse source problem, wave equation, modified method of fundamental solutions, sampling method
\end{abstract}


\section{Introduction}

The inverse problems for partial differential equations appear in various fields of science and engineering, and have been extensively studied in the past decades \cite{Isakov1998Inverse,Li2008Multilevel,Li2014Locating}. Among them, the inverse source problem, especially the identification of moving sources, has a wide range of applications such as under water sonar \cite{Lim1994On,Lim1994On2}, sound simulation and sound source localization \cite{Lassas2010Inverse,Matsumoto2003A}.

For the reconstruction of stationary sources, the inverse source problems with sources $\delta(t)g(x)$ that are delta-like in time and of limited oscillation in space, and sources $q(t)\delta(\partial G)$ that are oscillation in time and delta-like on the boundary of a region are considered in \cite{De2015An} and \cite{Ton2003An}, respectively. Uniqueness analysis related to the Helmholtz equation with phaseless data is shown in \cite{Zhang2018Uniqueness,Zhang2018Retrieval}. The stability analysis and identification of multiple point sources for the time-harmonic case are considered in \cite{Alves2009Iterative,Badia2011An}. The conditional stability estimate of the wave equation on a line related to the inverse source problem is provided by \cite{Cheng2002UNIQUENESS,Cheng2005The}. Multi-frequency inverse source problems are analyzed in \cite{Bao2010A,Bao2015aresursive,Li2016Increasing,Zhang2015Fourier}. Analysis of random sources can be seen in \cite{Bao2014An,Li2011An}. Time-dependent inverse source problems in elastodynamics are analyzed in \cite{Bao2017Inverse}.

For the reconstruction of a moving point source, direct identifications of the moving point source are studied in \cite{Nakaguchi2012An,wang2017mathematical}. Analysis of the moving point source when the velocity of the source is comparable to the speed of wave propagation can be seen in \cite{Garnier2015Super}. Matched-filter imaging method and correlation-based imaging for small fast moving debris with constant velocity are analyzed in \cite{Fournier2017Matched}. A gesture-based input technique with the electromagnetic wave is analyzed in \cite{Lijingzhi2018On}.

The method of fundamental solutions (MFS) is a meshless method which expands the solution utilizing the fundamental solution \cite{Ahmadabadi2009The,ChenB2019Method,Sun2014Modified,Wei2010Convergence}. The property of the fundamental solution, or the Green's function, is the theoretical basis of the MFS. However, the Green's function of the d'Alembert operator $c^{-2}\partial_{tt}-\Delta$ is
\begin{equation*}
G(x,t;s)=\dfrac{\delta(t-c^{-1}|x-s|)}{4\pi|x-s|},
\end{equation*}
where $c>0$ denotes the sound speed of the homogeneous background medium, $\partial_{tt}u=\frac{\partial^2 u}{\partial t^2}$, $\Delta$ is the Laplacian in $\mathbb{R}^3$, and $\delta$ is the Dirac delta distribution. Since the Green's function involves the Dirac delta distribution, the MFS is no longer feasible to solve the three dimensional wave equation. Unable to be applied directly, the Green's function of the d'Alembert operator usually appears in the time convolution
\begin{equation*}
G(x,t;s)*\lambda(t)=\frac{\lambda(t-c^{-1}|x-s|)}{4\pi|x-s|},
\end{equation*}
where $\lambda(t)$ is a signal function. One of the most popular application of $G(x,t;s)*\lambda(t)$ is that in the boundary integral equation method, which is a commonly used method \cite{ChenB2016time,guo2013toward,sayas2011retarded,Sun2017Indirect}. Therefore, instead of the Green's function, new bases $G(x,t;s)*\lambda(t)$ are employed in the modified method of fundamental solutions (MMFS) proposed in this paper. Moreover, the MMFS can be simplified to a simple sampling method at each time step to reconstruct a moving point source, in which the sampling method is a well-known method in the numerical computation of inverse problems \cite{Guo2016A,Li2013Two,Li2008Multilevel,Li2009Strengthened,Zhang2018Locating}.

The time convolution of the Green's function and the signal function is an invaluable tool for the analysis of the time domain scattering problems. Therefore, the MMFS has important significance in the theory of the time domain analysis. Moreover, the proposed methods are feasible to reconstruct both multiple stationary point sources and a moving point source. The numerical implementations of the proposed methods are simple, and extensive experiments are provided to show the effectiveness of the methods.

The outline of this paper is as follows. In Section 2, the inverse source problem with multiple stationary point sources is considered. The uniqueness result is provided and the MMFS is proposed. In Section 3, the MMFS is applied to the inverse source problem with a moving point source. Moreover, the method is simplified as a simple sampling method at each discrete time. In Section 4, numerical experiments are provided to show the effectiveness of the proposed methods. The conclusion remarks are given in Section 5.

\section{Reconstruction of stationary point sources}

Denote by $\Omega\subset \mathbb{R}^3$ a bounded convex open region. Consider the wave equation
\begin{equation}
\label{waveequation1}c^{-2}\partial_{tt}u(x,t)-\Delta u(x,t)=\lambda(t)\sum\limits_{j=1}\limits^{M}a_j\delta(x-s_j),\quad\quad x\in\mathbb{R}^3,\;t\in\mathbb{R},
\end{equation}
where $M\in\mathbb{N}^{*}$ is a positive integer, $s_j\in\Omega$ are stationary source points, and $a_j > 0$ are the intensities of the sources.

The source points $s_j$ are assumed to be mutually distinct. The signal function $\lambda(t)$ is assumed to be causal, which means $\lambda(t)=0$ for $t<0$. Thus the source term $f(t)=0$ for $t<0$, and the initial condition
\begin{equation}\label{initial}
u(\cdot,0)=\partial_t u(\cdot,0)=0\quad\quad\textup{in}\;\,\mathbb{R}^3
\end{equation}
is a direct conclusion of the causality.

The inverse source problem (P1) under consideration is: Determine the locations and intensities of the stationary point sources in \eqref{waveequation1} from the measurement data
\begin{equation}\label{boundary1}
u(x,t), \quad\quad  x\in\partial \Omega,\,t\in \mathbb{R}.
\end{equation}

The following lemma is needed to prove the uniqueness of the solution to the inverse source problem (P1).

\begin{Lemma}\label{lemma.distance}
Let $S:=\bigcup\limits_{j=1}\limits^{M}\left\{s_j\right\}$ be a set of points in a bounded convex open region $\Omega\subset \mathbb{R}^3$, where $M\in\mathbb{N}^{*}$. Assume that
$$d:=\min\limits_{x\in\partial \Omega,s\in S}|x-s|=\left|x_0-s_{k}\right|,$$
where $x_0\in \partial \Omega$ and $1\leq k\leq M$. Then
$$d_i':=\min\limits_{s\in S\backslash\left\{s_{k}\right\}}\left|x_0-s\right|>d.$$
\end{Lemma}

\begin{proof}
\begin{figure}
\centering
\includegraphics [width=0.40\textwidth]{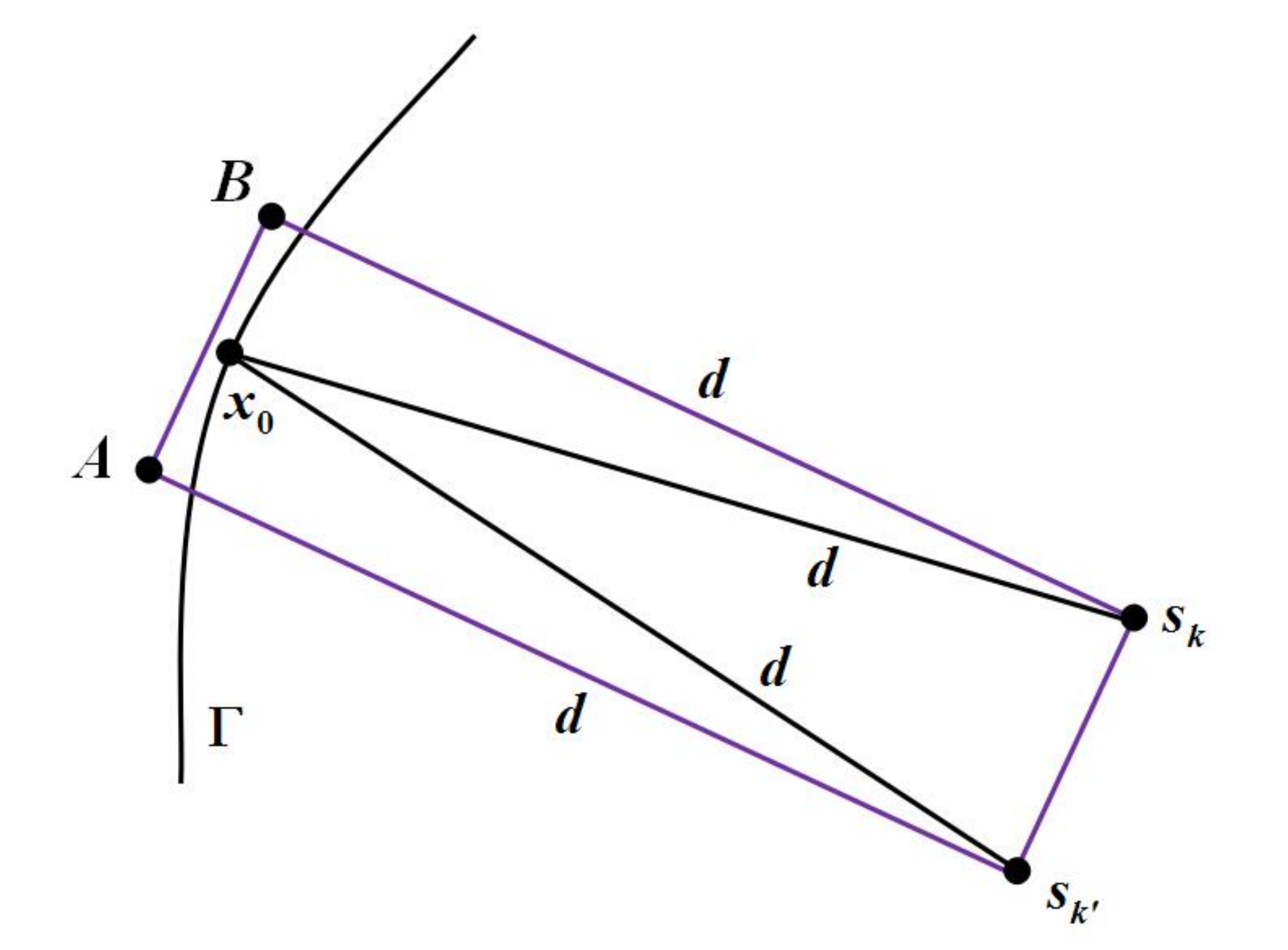}
\caption{Proof of Lemma \ref{lemma.distance}.}\label{fig-distance}
\end{figure}
It is obvious that $d'\geqslant d$. By reduction to absurdity, assume that $d'=\left|x_0-s_{k'}\right|=d$ for some $s_{k'}\in S\backslash\left\{s_{k}\right\}$. As is shown in Figure \ref{fig-distance}, construct a rectangle $s_{k}s_{k'}AB$ with $s_{k'}A=s_{k}B=d$, such that $x_0\in s_{k}s_{k'}AB$. Then $d=\min\limits_{x\in\partial \Omega,s\in S}|x-s|$ implies that $A,B\in\overline{\Omega}$. Thus $s_{k}s_{k'}AB\subset\overline{\Omega}$ since $\overline{\Omega}$ is a closed convex region. Spinning $s_{k}s_{k'}AB$ around the segment $s_{k}s_{k'}$, we get a cylinder $V\subset\overline{\Omega}$. Then there exists a $\varsigma>0$ such that $$B_{\varsigma}(x_0):=\left\{x:\left|x-x_0\right|<\varsigma\right\}\subset V\subset\overline{\Omega},$$
which is a contradiction to $x_0\in\partial \Omega$.
\end{proof}

Then we have the following uniqueness result.

\begin{Theorem}\label{prop.uniqueness}
Assume that $\Omega\subset \mathbb{R}^3$ is a bounded convex open region. Let
$$f_i(x,t)=\lambda^{(i)}(t)\sum\limits_{j=1}\limits^{M_i}a_j^{(i)}\delta\left(x-s_j^{(i)}\right), \quad i=1,2$$
be two source terms with $M_i\in\mathbb{N}^*$, $s_j^{(i)}\in \Omega$, $a_j^{(i)} > 0$ and $\lambda^{(i)}(t)\in C(\mathbb{R})$, such that the corresponding solutions to \eqref{waveequation1} for $f_1$ and $f_2$ are $u_1$ and $u_2$, respectively. Assume that $\lambda^{(i)}(t)$ are nontrivial causal functions and
$$u_1=u_2 \quad\quad  \text{on}\;\; \partial \Omega\times\mathbb{R}.$$
Then $M_1=M_2=M$, $s_j^{(1)}=s_{\pi(j)}^{(2)}$ and $a_j^{(1)}\lambda^{(1)}(t)=a_{\pi(j)}^{(2)}\lambda^{(2)}(t)$, $j=1,2,\ldots,M$ for some permutation $\pi(j)$ of $1,2,\ldots,M$.
\end{Theorem}

\begin{proof}
Denote $w:=u_1-u_2$. Then
\begin{align}
\label{waveequation-ball2}c^{-2}\partial_{tt}w-\Delta w=f_1-f_2&\quad\quad \textup{in}\;\,\mathbb{R}^3\times\mathbb{R},\\
\label{boundary-ball2}w=0&\quad\quad \textup{on}\;\,\partial \Omega\times\mathbb{R}.
\end{align}

Note that
\begin{equation*}
w_0(x,t):=\sum\limits_{j=1}\limits^{M_1}a_j^{(1)}G\left(x,t;s_j^{(1)}\right)*\lambda^{(1)}(t) -\sum\limits_{j=1}\limits^{M_2}a_j^{(2)}G\left(x,t;s_j^{(2)}\right)*\lambda^{(2)}(t)
\end{equation*}
is the unique causal solution (refer to Section 1.4 of \cite{sayas2011retarded} for the uniqueness) of the wave equation \eqref{waveequation-ball2}.

\begin{figure}
\centering
\includegraphics [width=0.33\textwidth]{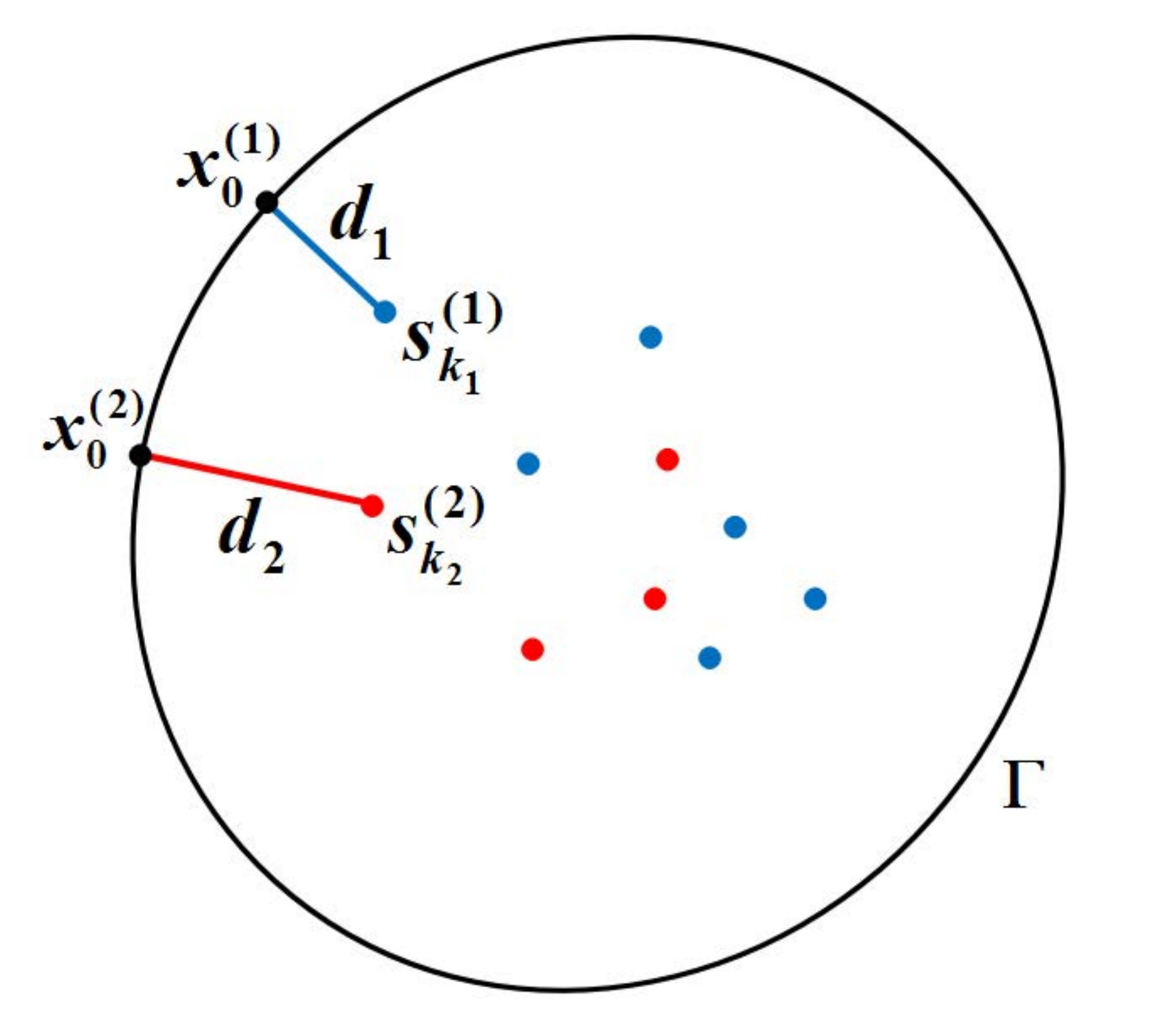}
\caption{Sketch of the measurement surface and the source points.}\label{fig-sketch}
\end{figure}

For the convenience of the expression, the rest of the proof is divided into five parts.

(\rNum{1}) As is shown in Figure \ref{fig-sketch}, denote
$$S^{(i)}:=\bigcup\limits_{j=1}\limits^{M_i}\left\{s_j^{(i)}\right\},\quad\quad i=1,2$$
and
\begin{equation}\label{eq-distance}
d_i:=\min\limits_{x\in\partial \Omega,s\in S^{(i)}}|x-s|=\left|x_0^{(i)}-s_{k_i}^{(i)}\right|,\quad\quad i=1,2,
\end{equation}
in which $x_0^{(i)}\in \partial \Omega$ and $1\leq k_i\leq M_i$. Notice that there may be several sets of points $x_0^{(i)}$ and $s_{k_i}^{(i)}$ which satisfy \eqref{eq-distance}. Nevertheless, the choice of the points would not affect the following proof.

Moreover, Lemma \ref{lemma.distance} implies that
$$d_i':=\min\limits_{s\in S^{(i)}\big\backslash\left\{s_{k_i}^{(i)}\right\}}\left|x_0^{(i)}-s\right|>d_i,\quad\quad i=1,2.$$

(\rNum{2}) Next, we are going to take into consideration of the causality. Since $\lambda^{(i)}(t)\in C(\mathbb{R})$ are nontrivial causal functions, we have
\begin{equation*}
\begin{cases}
\lambda^{(i)}(t)=0,\quad &t\leq t_i,\\
\lambda^{(i)}(t)\neq0,\quad &t_i<t<t_i+\tau,
\end{cases}
\quad\quad i=1,2
\end{equation*}
for some $t_i>0$ and $\tau>0$. Denote $t_i$ as the ``starting time'' of the signal $\lambda^{(i)}(t)$.

Assume that $t_1+c^{-1}d_1<t_2+c^{-1}d_2$. Then
$$w_0\left(x_0^{(1)},t\right)=a_{k_1}^{(1)} \frac{\lambda^{(1)}(t-c^{-1}d_1)}{4\pi d_1}, \quad\quad t\in(T_1,T_2),$$
where $T_1=t_1+c^{-1}d_1$ and $T_2=\min\{t_1+c^{-1}d_1',t_2+c^{-1}d_2\}$. Note that $\lambda^{(1)}(t-c^{-1}d_1)$ is nontrivial for $t\in(T_1,T_2)$. Then \eqref{boundary-ball2} implies $a_{k_1}^{(1)}=0$, which is a contradiction to $a_j^{(i)}>0$. Using the reduction to absurdity, we have $t_1+c^{-1}d_1\geq t_2+c^{-1}d_2$. Similarly, we can prove that $t_2+c^{-1}d_2\geq t_1+c^{-1}d_1$. Thus
\begin{equation}\label{equality1}
t_1+c^{-1}d_1 = t_2+c^{-1}d_2.
\end{equation}

(\rNum{3}) Denote
$$d_3 :=\min\limits_{s\in S^{(2)}}\left|x_0^{(1)}-s\right|.$$
Apparently $d_3\geq d_2$. Assume that $d_3 > d_2$. Then \eqref{equality1} implies $t_1+c^{-1}d_1<t_2+c^{-1}d_3$. A similar discussion as that in (\rNum{2}) leads to a contradiction, which means $d_3 = d_2$. Then \eqref{equality1} implies
\begin{equation}\label{equality2}
t_1+c^{-1}d_1 = t_2+c^{-1}d_3.
\end{equation}
Moreover, there is a point $s_{k_2''}^{(2)}\in S^{(2)}$ such that
$$\left|x_0^{(1)}-s_{k_2''}^{(2)}\right|=d_3=d_2=\min\limits_{x\in\partial \Omega,s\in S^{(2)}}|x-s|.$$

Therefore, we can reselect a new set of points $x_0^{(2)}$ and $s_{k_2}^{(2)}$ satisfying \eqref{eq-distance} with $x_0^{(2)}=x_0^{(1)}$ and $s_{k_2}^{(2)}=s_{k_2''}^{(2)}$.

Denote
$$d_3' :=\min\limits_{s\in S^{(2)}\big\backslash\left\{s_{k_2''}^{(2)}\right\}}\left|x_0^{(1)}-s\right|.$$
Then Lemma \ref{lemma.distance} implies $d_3'>d_3$. Assuming that $t_1>t_2$,  \eqref{equality2} implies $d_1<d_3$. For a point $x^*\in \left\{x\in\partial \Omega: 0<\left|x-x_0^{(1)}\right|<\frac{1}{2}\min\{d_1'-d_1,d_3'-d_3\}\right\}$, we have
$$\left|x^*-s_{k_2''}^{(2)}\right|-\left|x^*-s_{k_1}^{(1)}\right|<d_3-d_1.$$
Then \eqref{equality2} implies $t_2+c^{-1}\left|x^*-s_{k_2''}^{(2)}\right|<t_1+c^{-1}\left|x^*-s_{k_1}^{(1)}\right|$ and
\begin{equation*}
w_0(x^*,t)=-a_{k_2''}^{(2)} \frac{\lambda^{(2)}\left(t-c^{-1}\left|x^*-s_{k_2''}^{(2)}\right|\right)}{4\pi \left|x^*-s_{k_2''}^{(2)}\right|},\quad\quad t\in (T_3,T_4),
\end{equation*}
where $T_3=t_2+c^{-1}\left|x^*-s_{k_2''}^{(2)}\right|$ and $T_4=\min\left\{t_1+c^{-1}\left|x^*-s_{k_1}^{(1)}\right|, t_2+\frac{1}{2}c^{-1}(d_3+d_3')\right\}$. A similar discussion as that in (\rNum{2}) implies a contradiction, which means $t_1\leq t_2$. Similarly we have $t_2\leq t_1$. Then $t_1=t_2$ and $d_1=d_3$. Referring to the proof of Lemma \ref{lemma.distance}, we finally get $s_{k_1}^{(1)}=s_{k_2''}^{(2)}$.

(\rNum{4}) Define
$$\zeta(t)=a_{k_1}^{(1)}\lambda^{(1)}(t)-a_{k_2''}^{(2)}\lambda^{(2)}(t),$$
which is the real signal function of the source point $s_{k_1}^{(1)}=s_{k_2''}^{(2)}$. Then we assert that $\zeta(t)=0$. Otherwise, we can prove that
\begin{equation*}
\begin{cases}
\zeta(t)=0,\quad &t\leq t_3,\\
\zeta(t)\neq0,\quad &t_3<t<t_3+\tau'
\end{cases}
\end{equation*}
for some $t_3\geq t_2+c^{-1}\min\{d_1'-d_1,d_3'-d_3\}>t_2=t_1$ and $\tau'>0$. Notice that the set of source points $S:=S^{(1)}\cup S^{(2)}$ can be divided into two categories: $S^{(3)}$ such that the ``starting time'' of the signal function of $s\in S^{(3)}$ is $t_3$ and $S^{(4)}=S\setminus S^{(3)}$ such that the ``starting time'' of the signal function of $s\in S^{(4)}$ is $t_1$. If $S^{(4)}=\varnothing$, it is easy to get a contradiction. For $S^{(4)}\neq \varnothing$, a similar discussion as that in (\rNum{1})-(\rNum{3}) implies $t_3=t_1$, which is a contradiction to $t_3>t_1$.

Since we have proved $s_{k_1}^{(1)}=s_{k_2''}^{(2)}$ and $a_{k_1}^{(1)}\lambda^{(1)}(t)=a_{k_2''}^{(2)}\lambda^{(2)}(t)$, the wave field can be rewritten as
\begin{equation*}
w_0(x,t)=\sum_{\substack{j=1,...,M_1\\ j\neq k_1}} a_j^{(1)} G\left(x,t;s_j^{(1)}\right)*\lambda^{(1)}(t) -\sum_{\substack{j=1,...,M_2\\ j\neq k_2''}}a_j^{(2)}G\left(x,t;s_j^{(2)}\right)*\lambda^{(2)}(t).
\end{equation*}

(\rNum{5}) Assume that $M_1\neq M_2$, there is no harm to suppose that $M_1>M_2$. Following the procedure of (\rNum{1})-(\rNum{4}), we can get
\begin{equation*}
w_0(x,t)=\sum\limits_{j=1}\limits^{M_1-M_2} a_j^{(1')} G\left(x,t;s_j^{(1')}\right)*\lambda^{(1)}(t),
\end{equation*}
where $a_j^{(1')}>0$. Again, it is easy to get a contradiction. Then we have $M_1=M_2=M$. Following the procedure of (\rNum{1})-(\rNum{4}), we can finally get $s_j^{(1)}=s_{\pi(j)}^{(2)}$ and $a_j^{(1)}\lambda^{(1)}(t)=a_{\pi(j)}^{(2)}\lambda^{(2)}(t)$, $j=1,2,\ldots,M$ for some permutation $\pi(j)$ of $1,2,\ldots,M$.
\end{proof}

The classic MFS expands the solution utilizing the Green's function (refer to \cite{Sun2014Modified,Sun2017Indirect,Wei2010Convergence}). However, since the Green's function of the d'Alembert operator involves the Dirac delta distribution, the MFS is no longer feasible to solve the three dimensional wave equation. Hence, consider the expansion
\begin{equation}\label{MFS1}
u(x,t)=\sum\limits_{l=1}\limits^{N_z}c(z_l)G(x,t;z_l)*\lambda(t),
\end{equation}
where $N_z\in\mathbb{N}^*$, $z_l\in\Omega$ are the sampling points, and $c(z_l)$ are unknown coefficients to be computed.

The expansion \eqref{MFS1} leads to the first modified method of fundamental solutions (MMFS1). We introduce the following proposition concerning the MMFS1.

\begin{Proposition}\label{prop.algorithm1}
Assume that $\Omega\subset \mathbb{R}^3$ is a bounded convex open region. Let $u(x,t)$ be a causal wave field which solves
\begin{equation}\label{waveequation3}
c^{-2}\partial_{tt}u(x,t)-\Delta u(x,t)=\lambda(t)\sum\limits_{j=1}\limits^{M}a_j\delta(x-s_j), \quad\quad x\in\mathbb{R}^3,\;t\in\mathbb{R},
\end{equation}
where $M\in\mathbb{N}^*$, $s_j\in\Omega$, $a_j > 0$, and $\lambda(t)\in C(\mathbb{R})$ is a non-trivial causal signal function. Assuming that the sampling points $z_l\in\Omega,\,l=1,2,\ldots,N_z$ and a group of corresponding constants $c(z_l)$ satisfy
\begin{equation}\label{eq-alg1-continuous}
\sum\limits_{l=1}\limits^{N_z}c(z_l)G(x,t;z_l)*\lambda(t)=u(x,t),\quad \quad x\in\partial \Omega,\,t\in\mathbb{R},
\end{equation}
where $N_z\in\mathbb{N}^*$. Denote $S_d:=\{s_j\}_{j=1}^{M}$ and $Z_d:=\{z_l\}_{l=1}^{N_z}$. Then $S_d\subset Z_{d}$. Moreover,
\begin{equation*}
c(z_l)=
\begin{cases}
a_j,\quad\quad &z_l\in S_d,\\
0,\quad\quad &z_l\in Z_d\setminus S_d.
\end{cases}
\end{equation*}
\end{Proposition}

\begin{proof}
Notice that
\begin{equation}\label{realsolution1}
u(x,t)=\sum\limits_{j=1}\limits^{M}a_jG(x,t;s_j)*\lambda(t)
\end{equation}
is the unique causal solution of the wave equation \eqref{waveequation3}. Meanwhile, the wave field
\begin{equation*}
u'(x,t)=\sum\limits_{l=1}\limits^{N_z}c(z_l)G(x,t;z_l)*\lambda(t), \quad \quad x\in\mathbb{R}^3,\,t\in\mathbb{R}
\end{equation*}
is a causal solution of
\begin{equation*}
c^{-2}\partial_{tt}u(x,t)-\Delta u(x,t)=\lambda(t)\sum\limits_{l=1}\limits^{N_z}c(z_l)\delta(x-z_l), \quad\quad x\in\mathbb{R}^3,\;t\in\mathbb{R}.
\end{equation*}

Moreover, \eqref{eq-alg1-continuous} implies
\begin{equation*}
u'(x,t)=u(x,t),\quad \quad x\in\partial \Omega,\,t\in\mathbb{R}.
\end{equation*}
That is, $f_1(x,t)=\lambda(t)\sum\limits_{j=1}\limits^{M}a_j\delta(x-s_j)$ and $f_2(x,t)=\lambda(t)\sum\limits_{l=1}\limits^{N_z}c(z_l)\delta(x-z_l)$ are two solutions of the same inverse source problem. Then the conclusion follows from Theorem \ref{prop.uniqueness}.
\end{proof}

\begin{Remark}
It is a strong hypothesis that the chosen sampling points $z_l\in\Omega$ and the constants $c(z_l)$ solves \eqref{eq-alg1-continuous} for $x\in\partial \Omega$, $t\in\mathbb{R}$. Nevertheless, the numerical experiments in Section \ref{Sec_numerical} show the effectiveness of the MMFS1 even if the hypothesis is not satisfied.
\end{Remark}

The MMFS1 to solve the inverse source problem (P1) is shown in Algorithm $\textup{\RNum{1}}$. The numerical application of Algorithm $\textup{\RNum{1}}$ can be seen in Section \ref{Sec_numerical}.

\begin{table}[!bhtp]
  \small
  \centering
  \begin{tabular}{ll}
    \toprule
    \multicolumn{2}{c}{\textbf{Algorithm \RNum{1}.}~ MMFS1 to reconstruct stationary point sources}\\
    \midrule
    \textbf{Step 1}~ & Choose a convex region $\Omega$, a signal function $\lambda(t)$, an integer $M$ and the locations \\
     & $s_j~(j=1,\ldots,M)$ of the point sources. Collect the wave data $u(x_i,t_k)$ for the \\
     & sensing points $x_i\in\partial \Omega~(i=1,\ldots,N_x)$ and the discrete time steps $t_k\in [0,T]~$\\
     & $(k=1,\ldots,N_T)$, where $T$ is a chosen terminal time.\\
    \textbf{Step 2}~ & Choose a sampling region $D\subset\Omega$ such that $s_j\in D$ and $D\cap\partial \Omega=\varnothing$. Select a grid \\
    & of sampling points $z_l~(l=1,\ldots,N_z)$ in $D$. Compute $c(z_l)$ from \\
    & ~~~~~~~~$\sum\limits_{l=1}\limits^{N_z}c(z_l)(G*\lambda)(x_i,t_k;z_l)=u(x_i,t_k),\quad i=1,\ldots,N_x,\;k=1,\ldots,N_T$ \\
    & using the conjugate gradient method.  \\
    \textbf{Step 3}~ & Mesh $c(z_l)$ on the sampling grid. The locations of the point sources are given by \\
    & the locations of $z_l$ for which $c(z_l)$ are local maximum values.\\
    \bottomrule
  \end{tabular}
\end{table}

\section{Reconstruction of a moving point source}

In this section, the inverse source problem with a moving point source is considered. The wave equation is
\begin{equation}\label{waveequation2}
c^{-2}\partial_{tt}u(x,t)-\Delta u(x,t)=\lambda(t)\delta(x-s(t)),\quad\quad x\in\mathbb{R}^3,\,t\in [0,T],
\end{equation}
where $T>0$, $s:[0,T]\rightarrow \Omega$ signifies the smooth trajectory of the moving point source. Denote by
$$v(t)=\frac{\mathrm{d}s(t)}{\mathrm{d}t},\quad\quad t\in (0,T)$$
the instaneous velocity of the point source. Again, $\lambda(t)$ is causal and the initial condition follows from the causality.

The inverse source problem (P2) is: Determine the trajectory $s(t)$ of the moving point source in \eqref{waveequation2} from the measurement data
\begin{equation}\label{boundary2}
u(x,t), \quad\quad  x\in\partial \Omega,\,t\in [0,T].
\end{equation}

The MMFS1 is feasible to reconstruct stationary point sources. However, for a moving point source, the location of the point source changes over time. Thus the coefficients $c(z_l)$ in the MMFS1 should also depend on the time variable. Therefore, consider a new expansion
\begin{equation}\label{MFS2}
u(x,t)=\sum\limits_{l=1}\limits^{N_z}c(t;z_l)G(x,t;z_l)*\lambda(t),\quad \quad t\in [0,T],
\end{equation}
where $z_l\in\Omega$ are the sampling points and $c(\cdot;z_l)$ are unknown functions depending on $z_l$. The second modified method of fundamental solutions (MMFS2) is based on \eqref{MFS2}. The algorithm of MMFS2 is similar to Algorithm $\textup{\RNum{1}}$ except that the new expansion \eqref{MFS2} is employed and $c(t_k;z_l),\, l=1,\ldots,N_z$ should be computed respectively for each time step $t_k$, $k=1,\ldots,N_T$.

Notice that there is only one point source in this case. If $|v|=0$, we have $s(t_k)\equiv s_0$ for some $s_0\in\Omega$. Then Proposition \ref{prop.algorithm1} implies
\begin{equation*}
c(t_k,z_l)=
\begin{cases}
1,\quad\quad &z_l=s_0,\\
0,\quad\quad &\text{otherwise}.
\end{cases}
\end{equation*}
Then we expect that $G(x,t;s(t_k))*\lambda(t)$ is the approximation of $u(x,t_k)$ when $|v|$ is small. On this basis, define the indicator function
\begin{equation}\label{indicator}
I(z,t)=\left\|u_0(x,t;s(t))-G(x,t;z)*\lambda(t)\right\|_{\partial \Omega}^{-1},\quad\quad z\in\Omega,\,t\in [0,T],
\end{equation}
where $\|\cdot\|_{\partial \Omega}$ is the $L^2(\partial \Omega)$ norm with respect to $x$. We have the following theorem concerning the indicator function \eqref{indicator}.

\begin{Theorem}\label{prop.algorithm2}
Let $\Omega\subset \mathbb{R}^3$ be a bounded convex open region. Assume that $d_{\Omega}:=\sup\limits_{x,y\in\Omega}|x-y|\ll c$, $|v|\ll c$ and $\lambda,s\in C^1[0,T]$. Let $u_0(x,t)$ be the causal solution of the wave equation \eqref{waveequation2}. For any fixed $t\in[0,T]$, the indicator function \eqref{indicator} satisfies
$$I(z,t)\gg 1\quad \textup{when}\;z\rightarrow s(t).$$
\end{Theorem}

\begin{proof}
Note that when $|v(t)|<c$ for $t\in(0,T)$, the explicit solution to equation \eqref{waveequation2} is given by (refer to \cite{Nakaguchi2012An})
\begin{equation}\label{realsolution2}
u_0(x,t)=\frac{\lambda(\tau)}{4\pi|x-s(\tau)| \left(1-\frac{v(\tau)\cdot (x-s(\tau)) }{c|x-s(\tau)|}\right)},
\end{equation}
where the retarded time $\tau$ satisfies $t-\tau=c^{-1}|x-s(\tau)|$.

Under the assumptions $d_{\Omega}\ll c$, $|v|\ll c$ and $\lambda,s\in C^1[0,T]$, we assert that
\begin{equation*}
G(x,t;s(t))*\lambda(t)=\frac{\lambda(t-c^{-1}|x-s(t)|)}{4\pi|x-s(t)|}
\end{equation*}
is an approximation of the solution \eqref{realsolution2}. The proof of a similar conclusion can be seen in \cite{wang2017mathematical}. Though we use an arbitrary causal signal function $\lambda(t)$ instead of the time-harmonic signal $\lambda(t)=\sin(\omega_0 t)$ for some $\omega_0 >0$, and the function $\lambda(t-c^{-1}|x-s(t)|)$ is occupied in $G(x,t;s(t))*\lambda(t)$ in this paper instead of $\lambda(t)$, a similar discussion implies
\begin{equation*}
u_0(x,t)=G(x,t;s(t))*\lambda(t)+O(\varepsilon(t)),\quad\quad x\in\partial \Omega,\,t\in [0,T]
\end{equation*}
with some $0<\varepsilon(t)\ll 1$.

Then the smoothness of the function $G(x,t;z)*\lambda(t)$ with respect to $z$ implies the conclusion.
\end{proof}

Then the MMFS2 is in fact equivalent to a simple sampling method. The simplified scheme is shown in Algorithm $\textup{\RNum{2}}$. The numerical implement of Algorithm $\textup{\RNum{2}}$ is shown in Section \ref{Sec_numerical}.

\begin{table}[!htbp]
  \small
  \centering
  \begin{tabular}{ll}
    \toprule
    \multicolumn{2}{c}{\textbf{Algorithm \RNum{2}.}~ The simplified scheme to reconstruct a moving point source }\\
    \midrule
    \textbf{Step 1}~ & Choose a convex region $\Omega$, a signal function $\lambda(t)$ and the trajectory $s(t),t\in[0,T]$ \\
     & of the moving point source. Collect the wave data $u(x_i,t_k)$ on the sensing points \\
     & $x_i\in\partial \Omega~(i=1,\ldots,N_x)$ and the discrete time steps $t_k~(k=1,\ldots,N_T)$. \\
    \textbf{Step 2}~ & Choose a sampling region $D\subset\Omega$ such that $s(t)\subset D$ and $D\cap\partial \Omega=\varnothing$. Select a grid \\
    & of sampling points $z_l~(l=1,\ldots,N_z)$ in $D$. For each time step $t_k$, compute \\
    & ~~~~~~~~~~~~$I(z_l,t_k)=\left(\sum\limits_{i=1}^{N_x}\big((G*\lambda)(x_i,t_k;z_l)-u(x_i,t_k) \big)^2\right)^{-1/2}.$ \\
    \textbf{Step 3}~ & For each time step $t_k$, the location $s(t_k)$ of the point source is approximated by \\
    & the location of $z_l$ for which $I(z_l,t_k)$ is the global maximum value. \\
    \bottomrule
  \end{tabular}
\end{table}

\section{Numerical examples}\label{Sec_numerical}

In this section, we consider the numerical implementation of the proposed algorithms. The radiated field is collected for $t\in [0,T]$, where $T$ is the terminal time. The time discretization is
$$t_k=k\dfrac{T}{N_T},\quad k=0,1,\ldots,N_T,$$
where $N_T\in\mathbb{N}^*$. Random noises are added to the data with
$$u_{\epsilon}=(1+\epsilon r)u,$$
where $\epsilon>0$ is the noise level and $r$ are uniformly distributed random numbers in $[-1,1]$.

In all the experiments, the signal function $\lambda(t)$ is chosen as
$$\lambda(t)=
\begin{cases}
0,\quad\quad&t<0,\\
\sin(10t)\mathrm{e}^{-0.3(t-3)^2},\quad\quad&t\geq 0.
\end{cases}$$
The signal function $\lambda(t)$ and its Fourier spectrum can be seen in Figure \ref{fig002}.

\begin{figure}
\centering
\subfigure[]{
\includegraphics[scale=0.36]{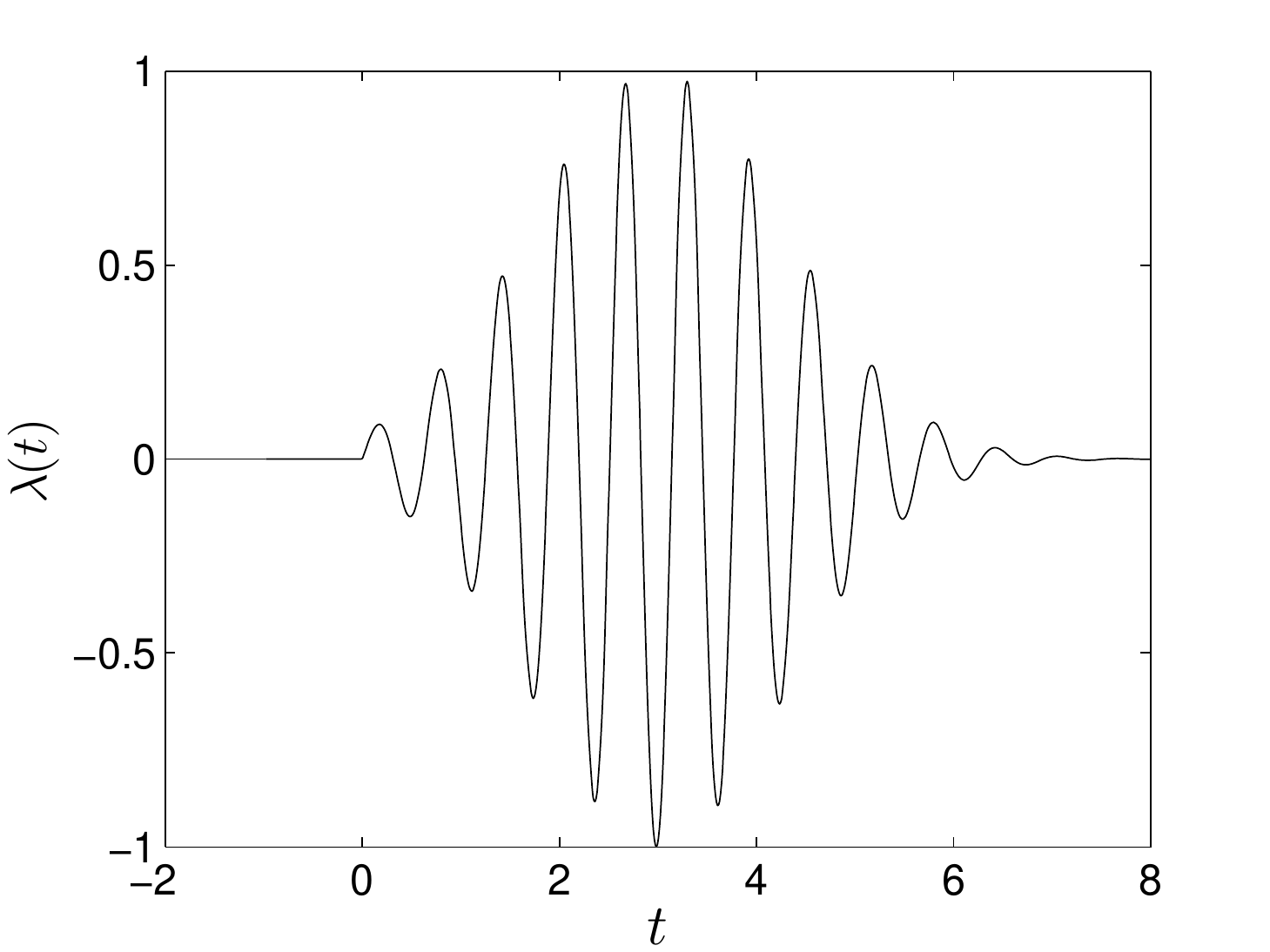}
}
\subfigure[]{
\includegraphics[scale=0.36]{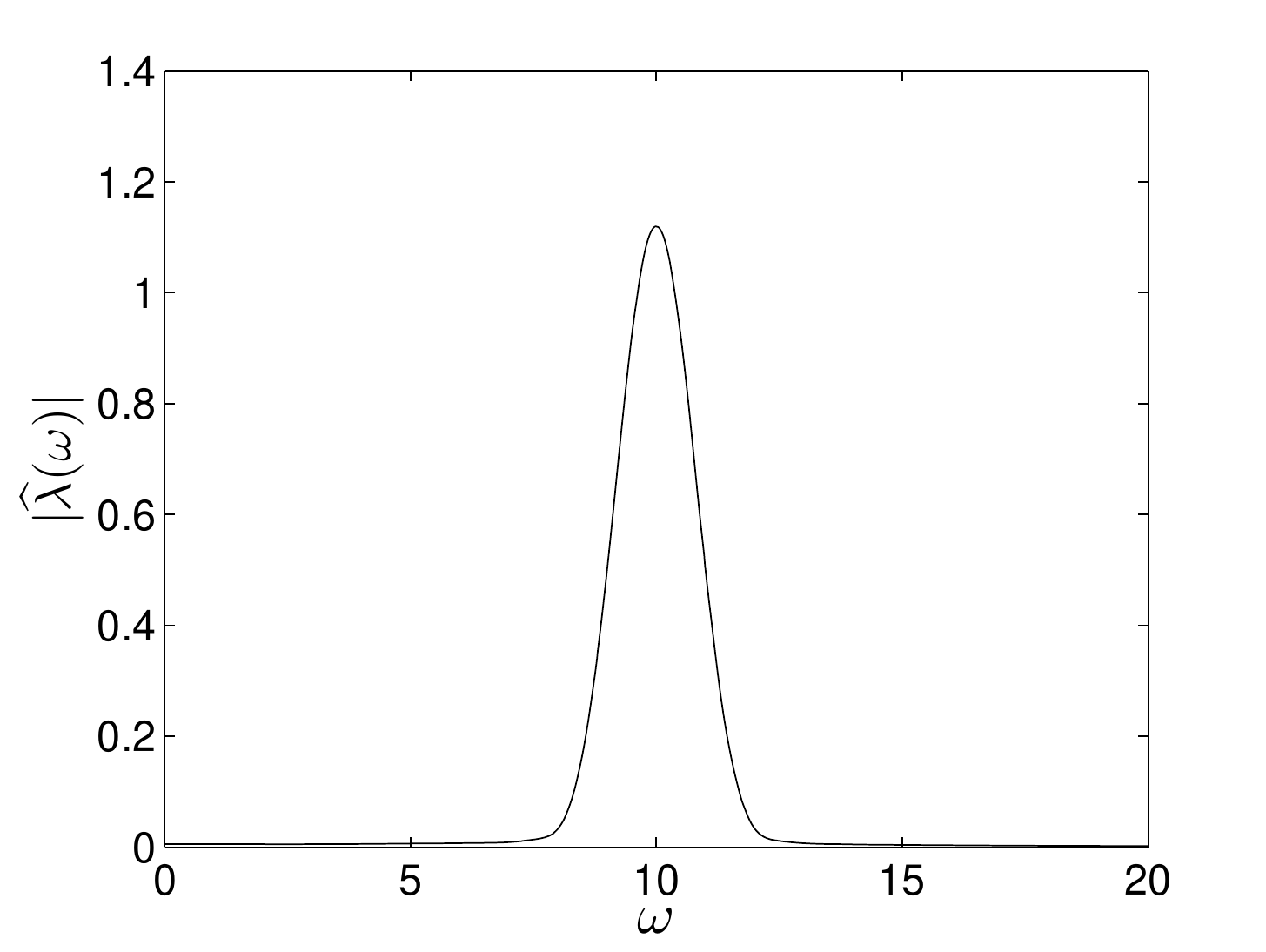}
}
\caption{(a) The pulse function $\lambda(t)=\sin (10t) \mathrm{e}^{-0.3(t-3)^2}$. (b) The Fourier spectrum $|\hat{\lambda}(\omega)|$. }\label{fig002}
\end{figure}

\subsection[]{Reconstruction of multiple static point sources}\label{subsec-numer1}

Algorithm $\text{\RNum{1}}$ is employed for the reconstruction of multiple static point sources. The synthetic data $u(x,t)$, $x\in\partial \Omega$, $t\in[0,T]$ is given by the analytic solution \eqref{realsolution1}. We choose $c=1$, $T=15$ and $N_T=64$ in this subsection.

\vspace{2mm}

\noindent\textbf{Example 1.} In this example, the reconstruction of the stationary point sources located at $(0,-1,1)$, $(0,1,-1)$, $(1,-1,0)$ and $(1,0,-1)$ with the same intensity is considered. The sensing points are chosen as
\begin{equation}\label{sensors}
x(i,j)=(5\sin\varphi_i\cos\theta_j,\,5\sin\varphi_i\sin\theta_j,\,5\cos\varphi_i)
\end{equation}
with $\varphi_i=\frac{2i-1}{16}\pi$, $i=1,2,\ldots,8$ and $\theta_j=\frac{j}{4}\pi$, $j=0,1,\ldots,7$.
The sampling points are $21\times 21\times 21$ uniform discrete points in $[-2,2]\times[-2,2]\times[-2,2]$. The reconstructions are shown in Figure \ref{fig005}.

\begin{Remark}
To facilitate the 3D visualization, we add 2D projections in some of the 3D figures in this paper.
\end{Remark}

\begin{figure}
\centering
\subfigure[]{
\includegraphics[scale=0.30]{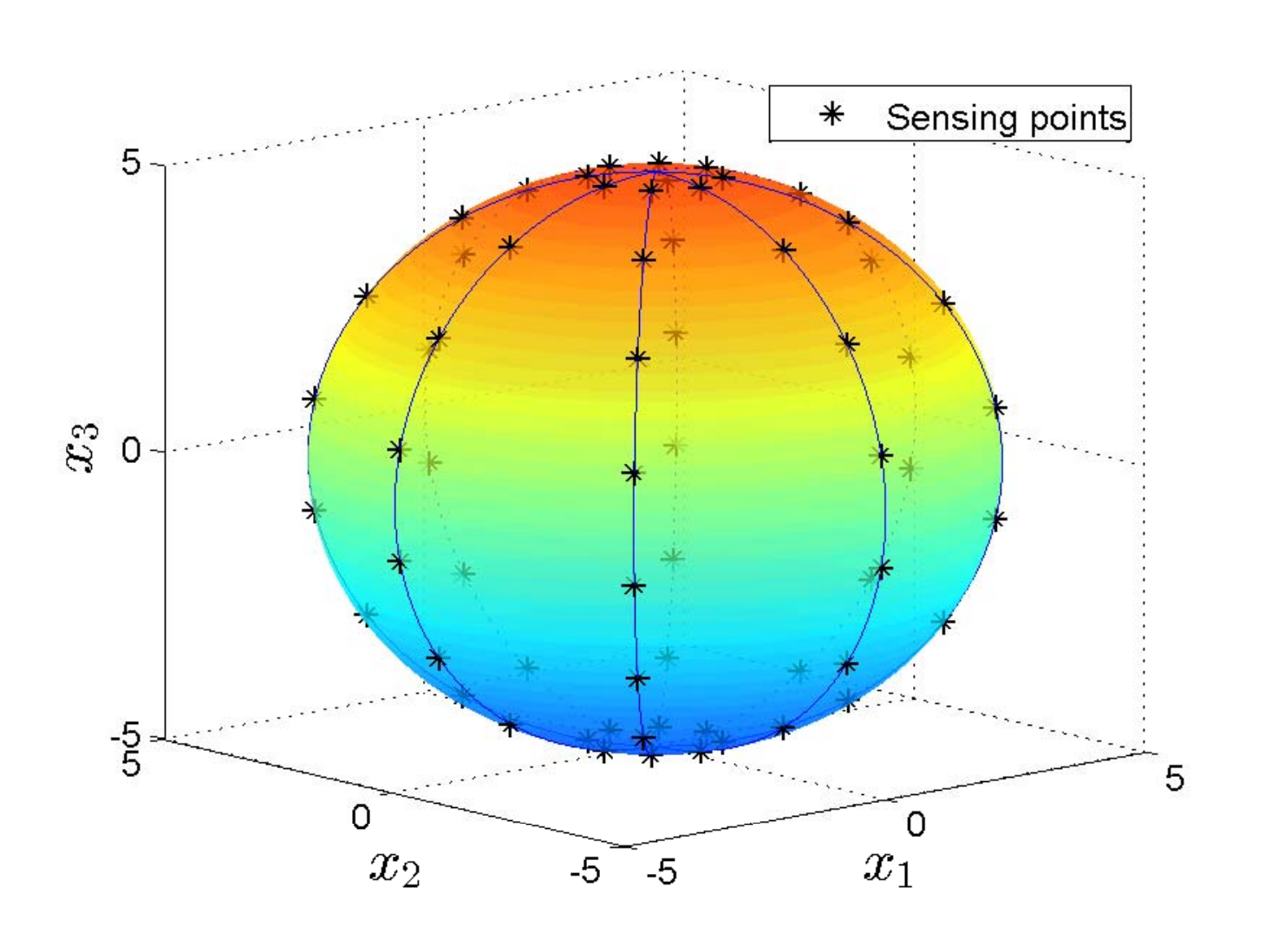}
}
\subfigure[]{
\includegraphics[scale=0.265]{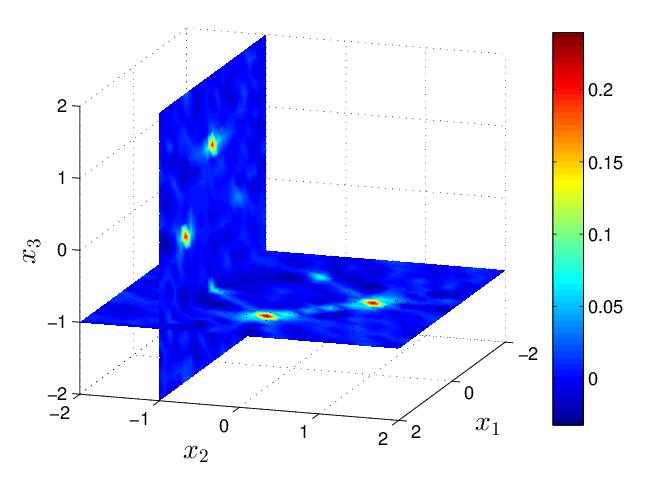}
}

\subfigure[]{
\includegraphics[scale=0.265]{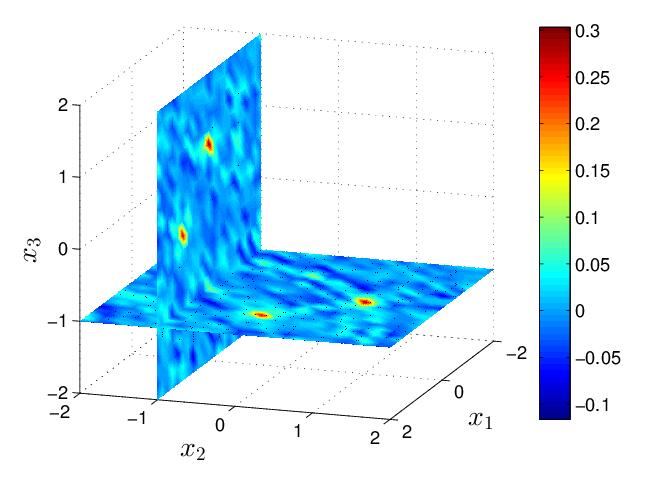}
}
\subfigure[]{
\includegraphics[scale=0.30]{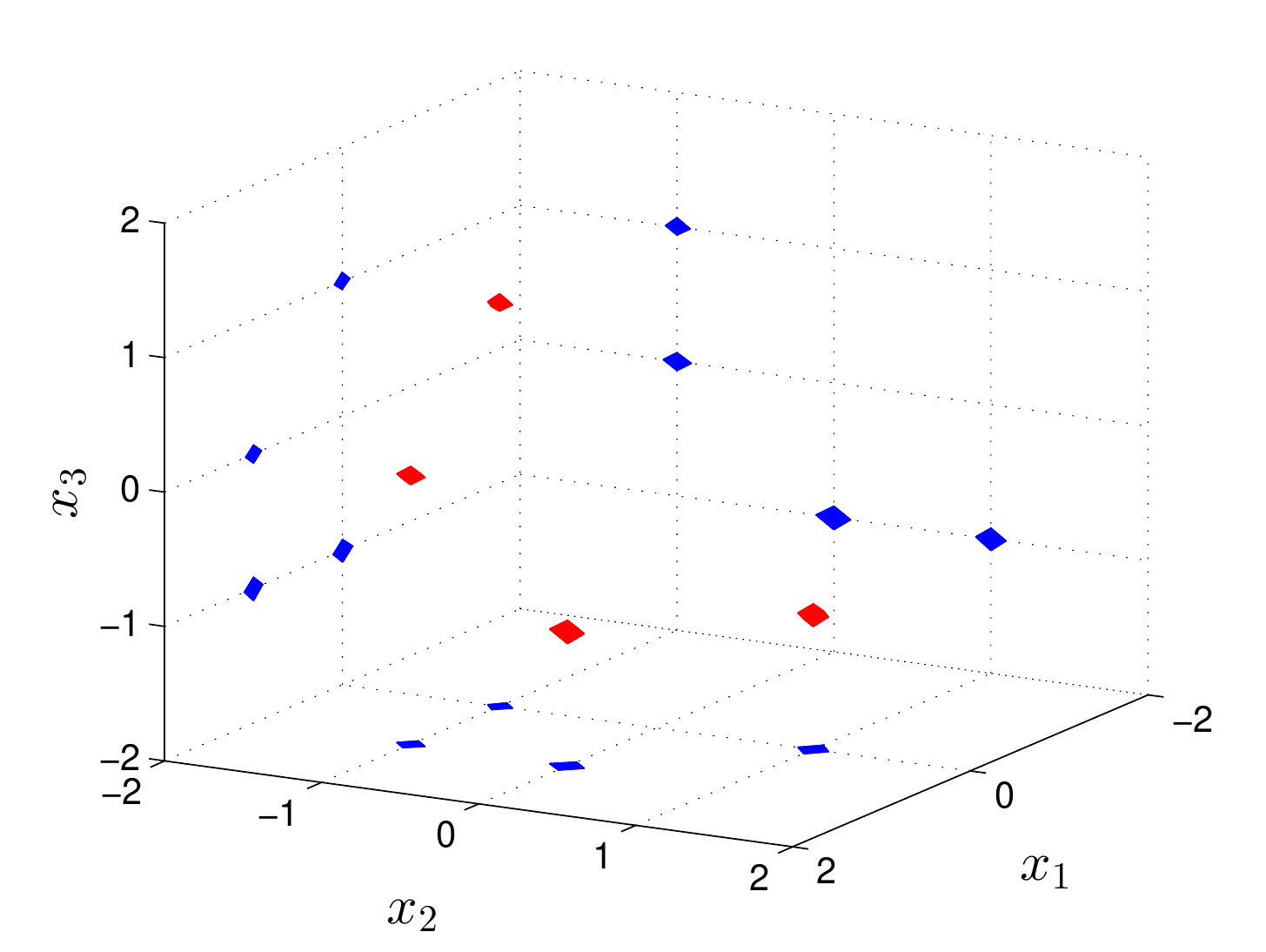}
}
\caption{Reconstruction of 4 stationary point sources with the same intensity (Example 1). (a) Location of the sensors. (b) Reconstruction of the point sources, $\epsilon=1\%$. (c) Reconstruction of the point sources, $\epsilon=5\%$. (d) The isosurface of $c(z_l)=0.7\max\{c(z_l)\}$, $\epsilon=5\%$. } \label{fig005}
\end{figure}

\noindent\textbf{Example 2.} We investigate the reconstruction of point sources with different intensities in this example. The source points are chosen as $(0,1,0)$, $(-1,-1,0)$, $(-1,-1.2,0)$, $(1,0.5,0)$, $(-0.5,0.5,0)$ and $(1.5,-1,0)$ with relative intensities $2$, $3$, $2$, $4$, $3$ and $3$, respectively. The sampling points are chosen as $23\times 23\times 23$ uniform discrete points in $[-2,2]\times[-2,2]\times[-2,2]$. The sensors are chosen as all the sensors in \eqref{sensors}, the left half of the sensors with $i=1,2,\ldots,8$, $j=0,1,\ldots,3$ and the upper half with $i=1,2,\ldots,4$, $j=0,1,\ldots,7$, respectively in three cases. The reconstructions are shown in Figure \ref{fig003}.

\begin{figure}
\centering
\subfigure[]{
\includegraphics[scale=0.33]{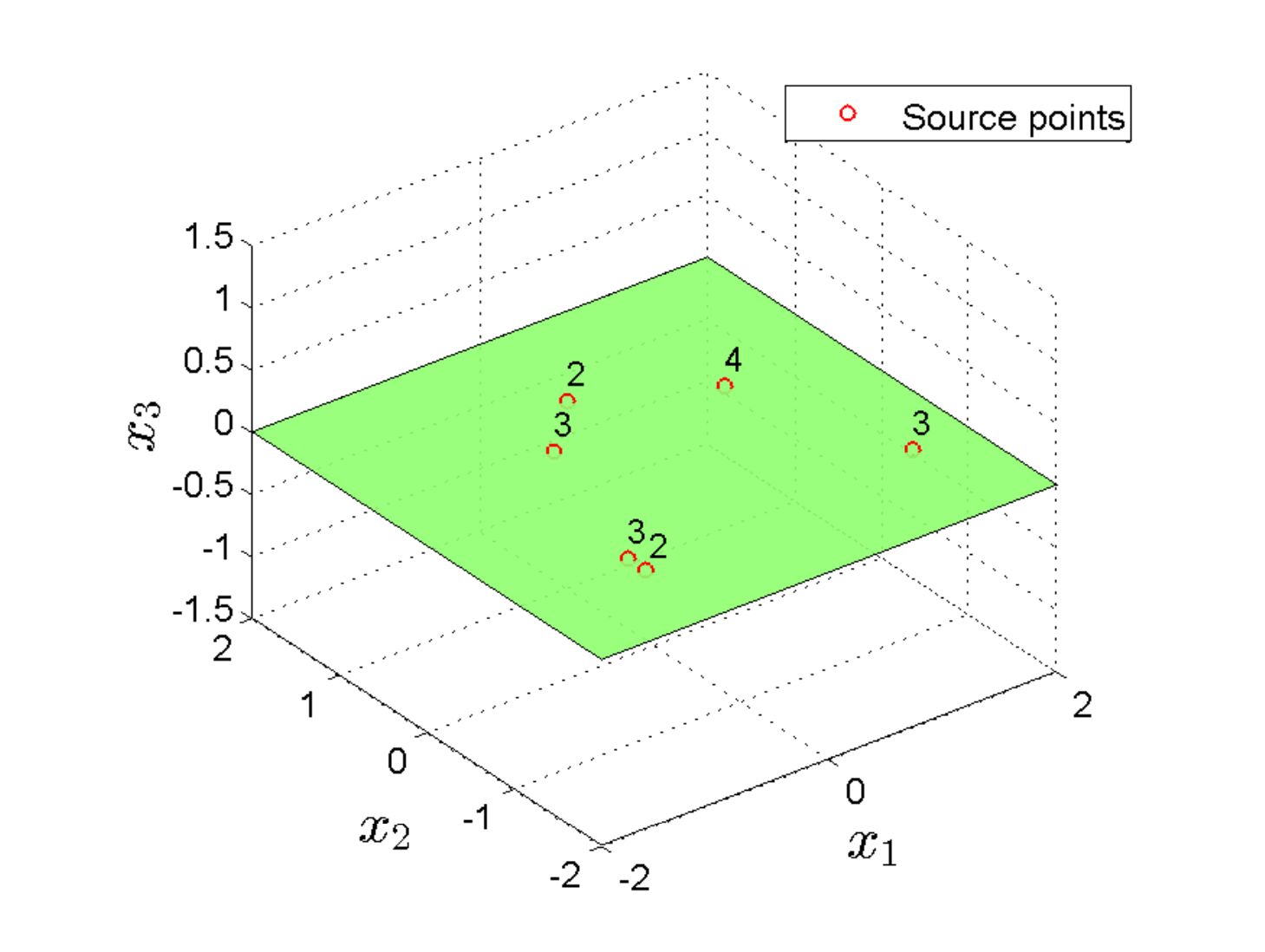}
}
\subfigure[]{
\includegraphics[scale=0.33]{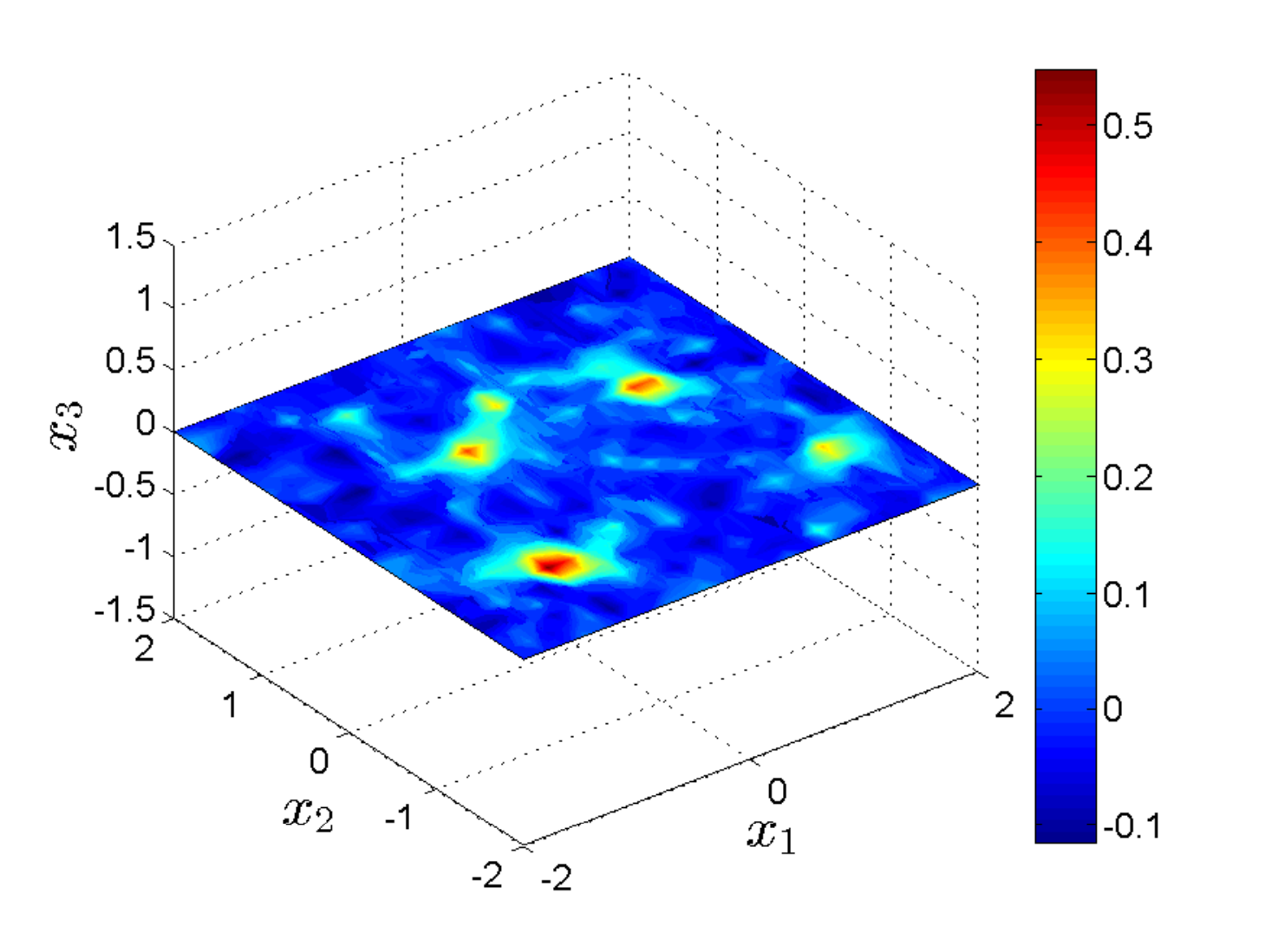}
}

\subfigure[]{
\includegraphics[scale=0.33]{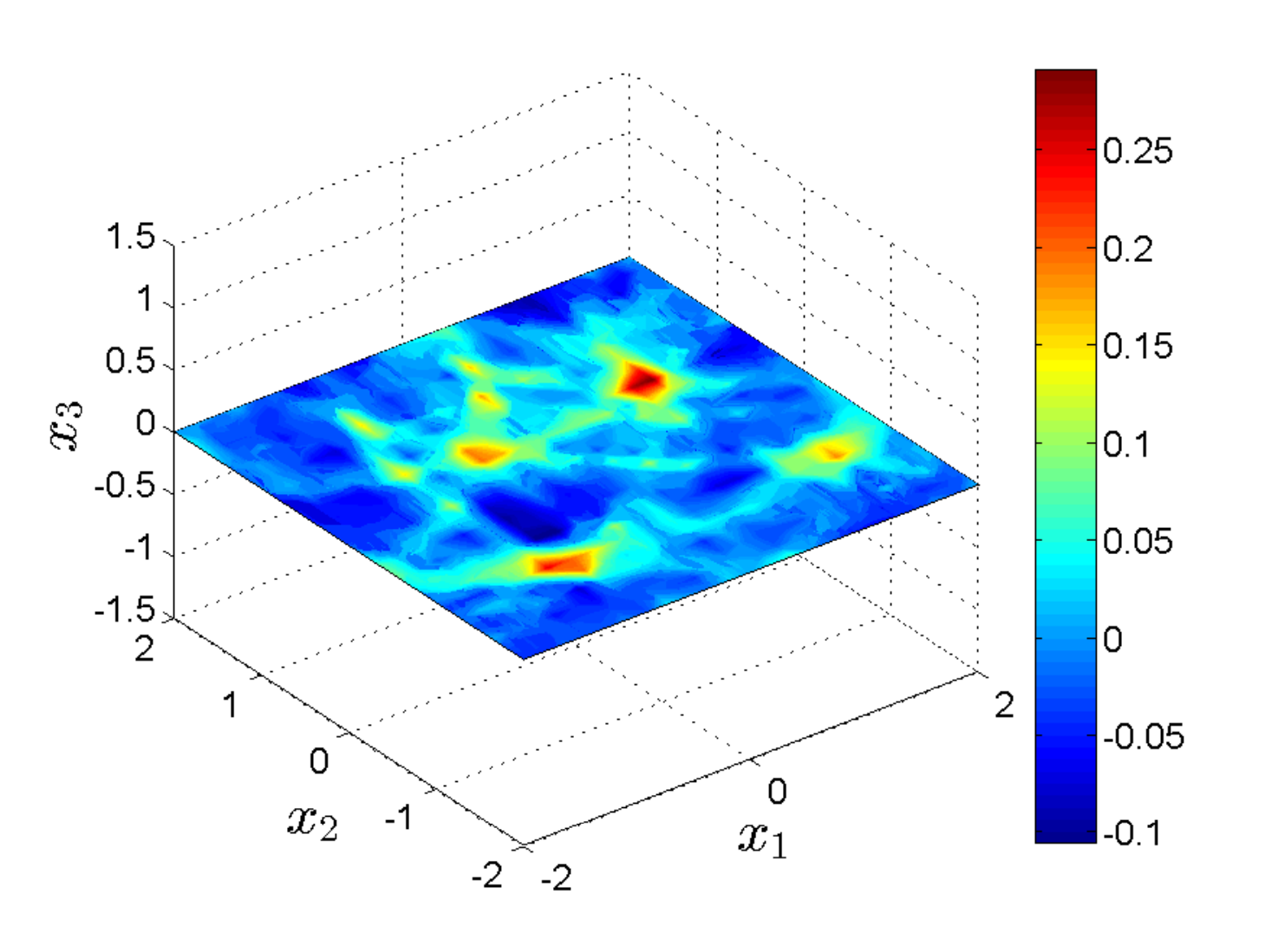}
}
\subfigure[]{
\includegraphics[scale=0.33]{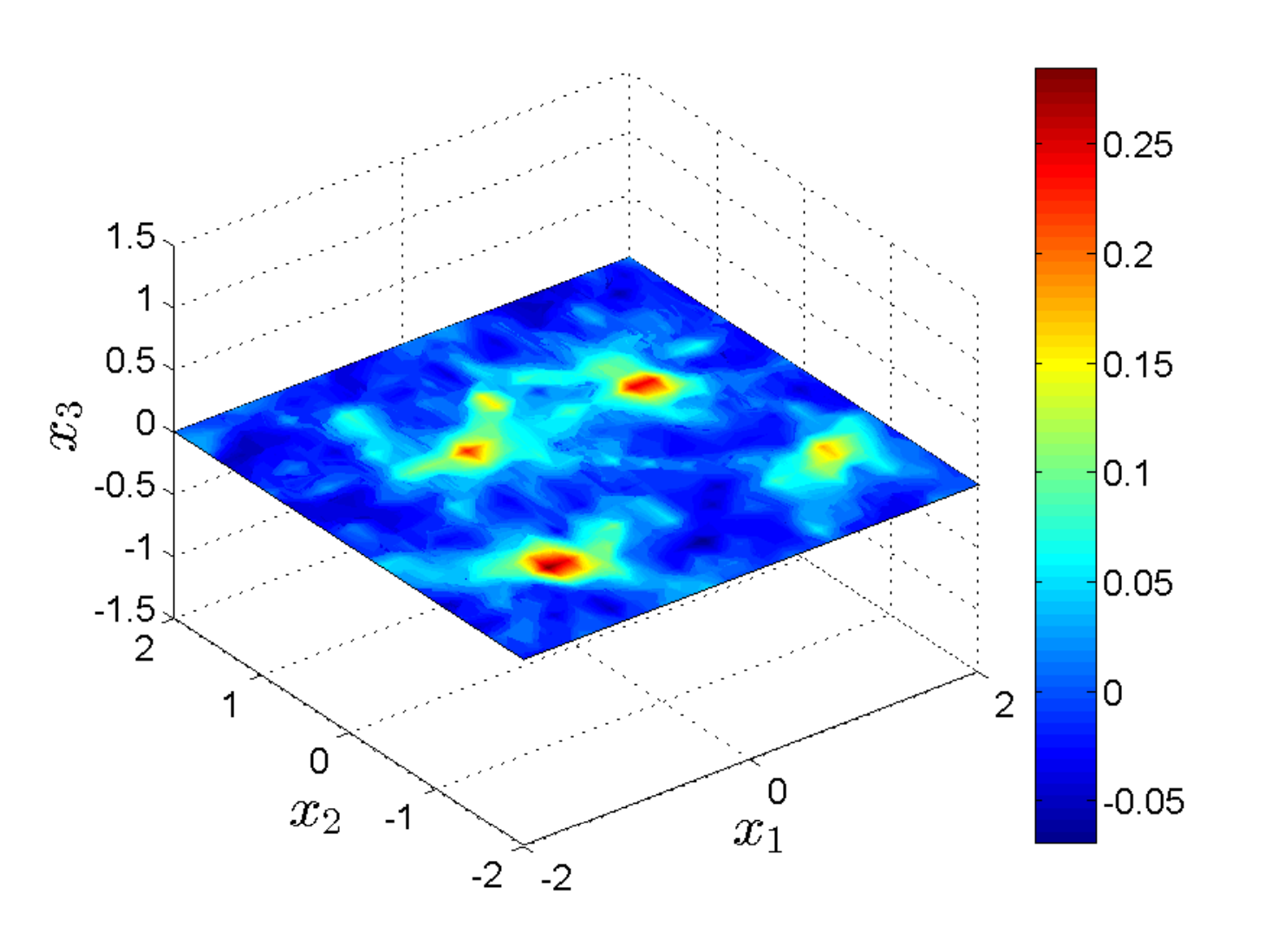}
}
\caption{Reconstruction of 6 stationary point sources with different intensities (Example 2). (a) Sketch of the example. (b) Reconstruction with all the sensors, $\epsilon=1\%$. (c) Reconstruction with left half of the sensors, $\epsilon=1\%$. (d) Reconstruction with upper half of the sensors, $\epsilon=1\%$. } \label{fig003}
\end{figure}

As is shown in Figure \ref{fig003}, the proposed method is feasible to reconstruct point sources with different intensities. The specific data of the reconstruction is given by the following procedure:\vspace{2mm}\\
{\small
(1) Compute $c(z_l)$ and save the data as $C(i,j,k),\,i,j,k=1,...,23$, corresponding to the sampling points $z(i,j,k)=\left(-2+\frac{2(i-1)}{11},-2+\frac{2(j-1)}{11},-2+\frac{2(k-1)}{11}\right)$, respectively. Denote $n=1$. \\
(2) Find a global maximum of $C(i,j,k)$ and the corresponding maximum point $z(i_n,j_n,k_n)$. The intensity of the point $z(i_n,j_n,k_n)$ is given by
$$T(i_n,j_n,k_n)=\sum_{\substack{i=i_n-1,...,i_n+1\\j=j_n-1,...,j_n+1\\k=k_n-1,...,k_n+1}}C(i,j,k).$$
(3) If $T(i_n,j_n,k_n)<1$, end the procedure.\\
(4) If $T(i_n,j_n,k_n)>1$, the corresponding point $z_l(i_n,j_n,k_n)$ is regarded as a source point with the intensity $T(i_n,j_n,k_n)$. \\
(5) Denote $C(i,j,k)=0$ for $i=i_n-1,...,i_n+1$, $j=j_n-1,...,j_n+1$ and $k=k_n-1,...,k_n+1$. Redefine $n=n+1$ and go back to step (2).
}

\vspace{2mm}

The specific data is shown in Table 1. The error of the location is mainly caused by the discretization precision of the sampling region. Since the point sources No. 2 and No. 3 are too close to each other, only one source point is reconstructed with the superimposition of the intensities.

\begin{table}[tp]
  \centering
  \fontsize{10}{14}\selectfont
  \label{table-reconstruction}
    \begin{tabular}{|c|c|c|c|c|}
    \hline
    \multirow{2}{*}{No.}&
    \multicolumn{2}{c|}{The point sources}&\multicolumn{2}{c|}{The reconstructions}\cr
    \cline{2-5}
    &\quad\quad Location\quad\quad\quad &\quad Intensity\quad\quad &\quad\quad Location\quad\quad\quad &\quad Intensity\quad\quad \cr
    \hline
    1&$(0,1,0)$&2 &$(0,0.91,0)$&2.01 \cr\hline
    2&$(-1,-1,0)$&3 &$(-1.09,-1.09,0)$&5.17\cr\hline
    3&$(-1,-1.2,0)$&2 &Null&Null\cr\hline
    4&$(1,0.5,0)$&4 &$(0.91,0.55,0)$&4.31\cr\hline
    5&$(-0.5,0.5,0)$&3 &$(-0.55,0.55,0)$&3.42\cr\hline
    6&$(1.5,-1,0)$&3 &$(1.45,-0.91,0)$&3.13\cr\hline
    \end{tabular}
  \caption{Reconstruction of the locations and intensities of the point sources.}
\end{table}

\subsection[]{Reconstruction of a moving point source}\label{subsec-numer3}

This subsection is concerned with the reconstruction of a moving point source. Numerical scheme based on Algorithm $\textup{\RNum{2}}$ is employed. The synthetic data $u(x,t)$, $x\in\partial \Omega$, $t\in[0,T]$ is given by the analytic solution \eqref{realsolution2}. We choose $c=340$ and $T=2\pi$ in this subsection.

\vspace{2mm}

\noindent\textbf{Example 3.} In this example, we consider the reconstruction of arbitrary trajectory of a moving source in $\mathbb{R}^3$. The sensors are chosen the same as that in Example 1. The sampling points are chosen as $51\times 51\times 51$ uniform discrete points in $[-3,3]\times[-3,3]\times[-3,3]$. We choose $N_T=64$ in this experiment. The reconstructions $s'(t_k)$ of the locations $s(t_k)$ for $k=1,2,...,N_T$ are considered.

The trajectories of the moving source are chosen as $s_1(t)=(2+0.3\cos 3t)(\cos t, \sin t, 0)$ and $s_2(t)=2(\sin {2t}, \cos {2t}, \frac{t}{\pi}-1)$, respectively in two cases. The reconstructions can be seen in Figure \ref{fig006} and Figure \ref{fig007}, respectively.

As is shown in Figure \ref{fig006}(b), the reconstructions $s'_1(t_k)$ is close to the trajectory $s_1(t)$ except for several discrete points. The error given by the Euclidean distance $|s_1(t_k)-s'_1(t_k)|$ at each discrete time steps $t_k$ can be seen in Figure \ref{fig006}(c). The error becomes large when the signal intensity $\lambda(t_k)$ is near zero. Therefore, the following modification is provided after the reconstruction:\vspace{2mm}\\
{\small
(1) If $|\lambda(t_1)|<10^{-4}$, redefine $s'(t_1)=2s'(t_2)-s'(t_3)$. \\
(2) If $|\lambda(t_{N_T})|<10^{-4}$, redefine $s'(t_{N_T})=2s'(t_{N_T-1})-s'(t_{N_T-2})$. \\
(3) If $|\lambda(t_k)|<10^{-4}$ for any $k=2,\ldots,N_T-1$, redefine $s'(t_k)=\frac{1}{2}(s'(t_{k-1})+s'(t_{k+1}))$.
}

\vspace{2mm}

The modified reconstruction and the corresponding error are shown in Figure \ref{fig006}(d) and Figure \ref{fig006}(e), respectively. As we can see from Figure \ref{fig006}(e), the error $|s(t_k)-s'(t_k)|$ is small at each time steps after the modification. Therefore, similar modifications are applied to all the experiments in the rest of this subsection.

The smooth reconstruction of the trajectory is given by the post-processing of the data $s'(t_k)$ by a Fourier approximation.  The truncated Fourier expansion of order $N\in\mathbb{N}^{*}$ is employed such that
$$s(t)=a_0+\sum\limits_{n=1}\limits^{N}(a_n \cos{nt}+b_n \sin{nt}),$$
where
\begin{align*}
a_0=&\frac{1}{N_T}\sum\limits_{k=1}\limits^{N_T}s'(t_k),\\
a_n=&\frac{2}{N_T}\sum\limits_{k=1}\limits^{N_T}s'(t_k)\cos{nt_k},\quad\quad n=1,2,\ldots,N,\\
b_n=&\frac{2}{N_T}\sum\limits_{k=1}\limits^{N_T}s'(t_k)\sin{nt_k},\quad\quad n=1,2,\ldots,N.
\end{align*}

Fourier expansion of order 5 is employed to get the smooth reconstruction in this example.

\begin{Remark}
An important component of the error is caused by the discretization precision of the sampling region. The error $|s(t_n)-s'(t_n)|<0.2$ coincides with the $51\times 51\times 51$ uniform discretization of the sampling region $[-3,3]\times[-3,3]\times[-3,3]$.
\end{Remark}

\begin{figure}
\centering
\subfigure[]{
\includegraphics[scale=0.31]{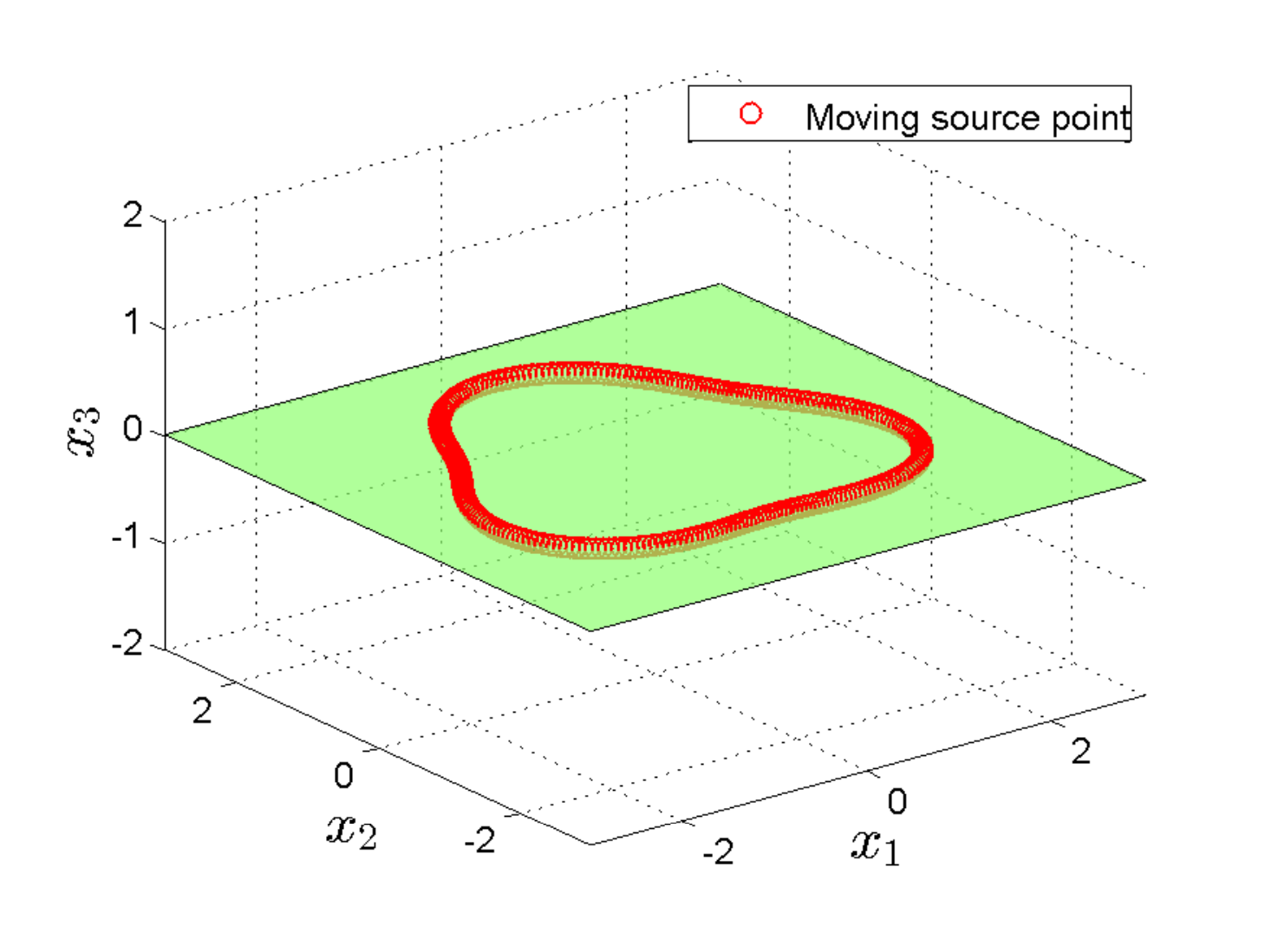}
}
\subfigure[]{
\includegraphics[scale=0.31]{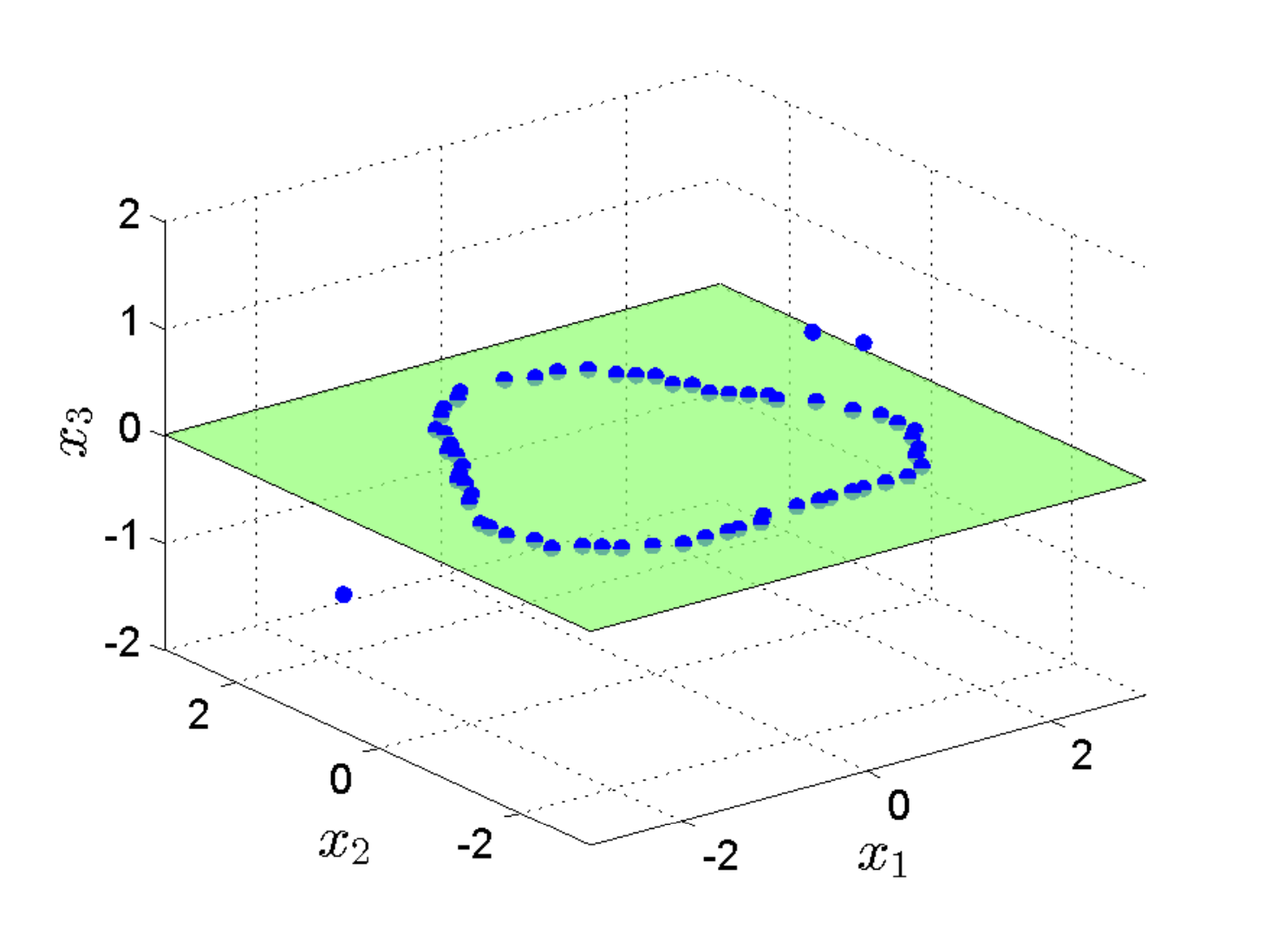}
}
\subfigure[]{
\includegraphics[scale=0.31]{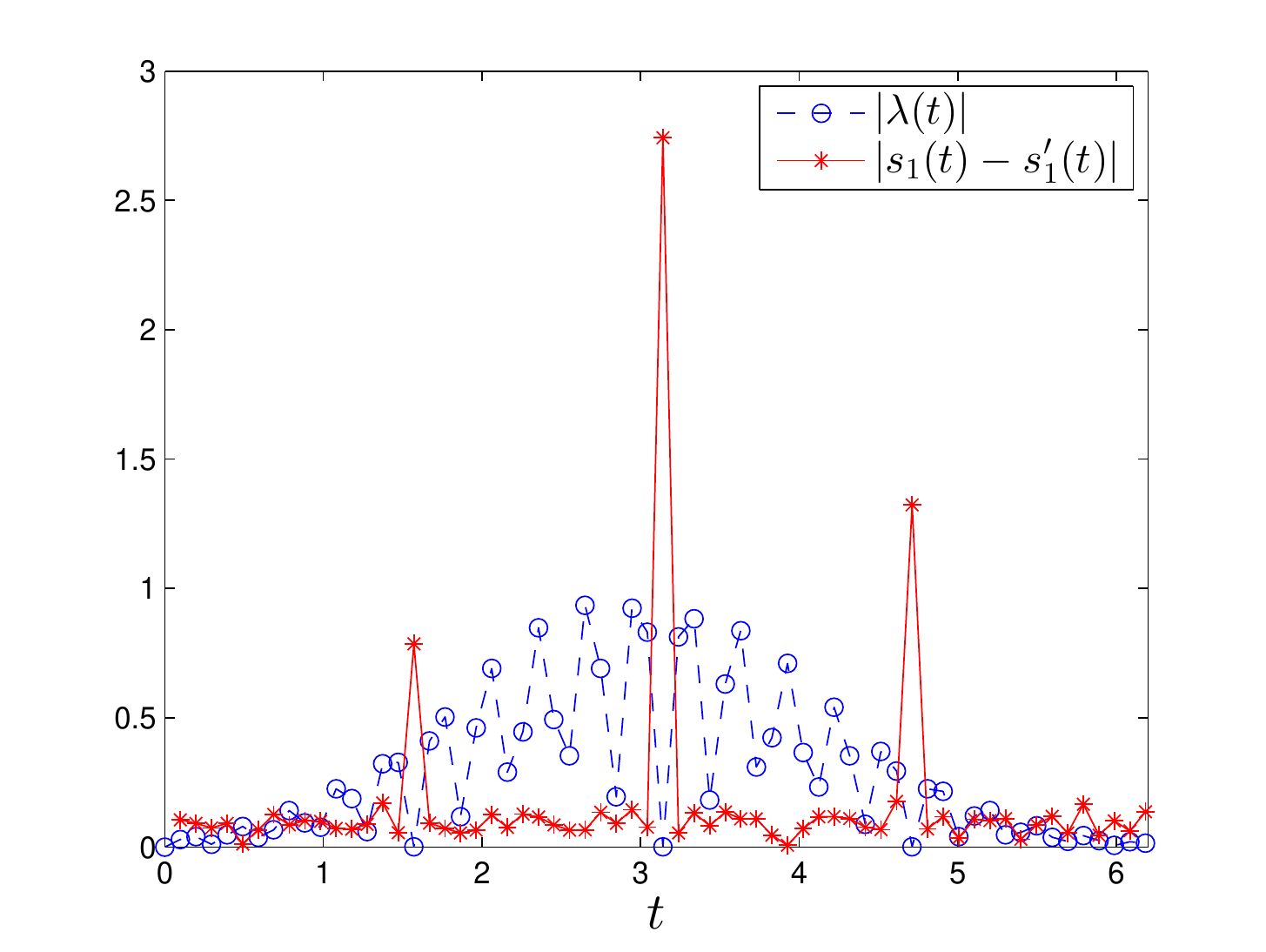}
}

\subfigure[]{
\includegraphics[scale=0.31]{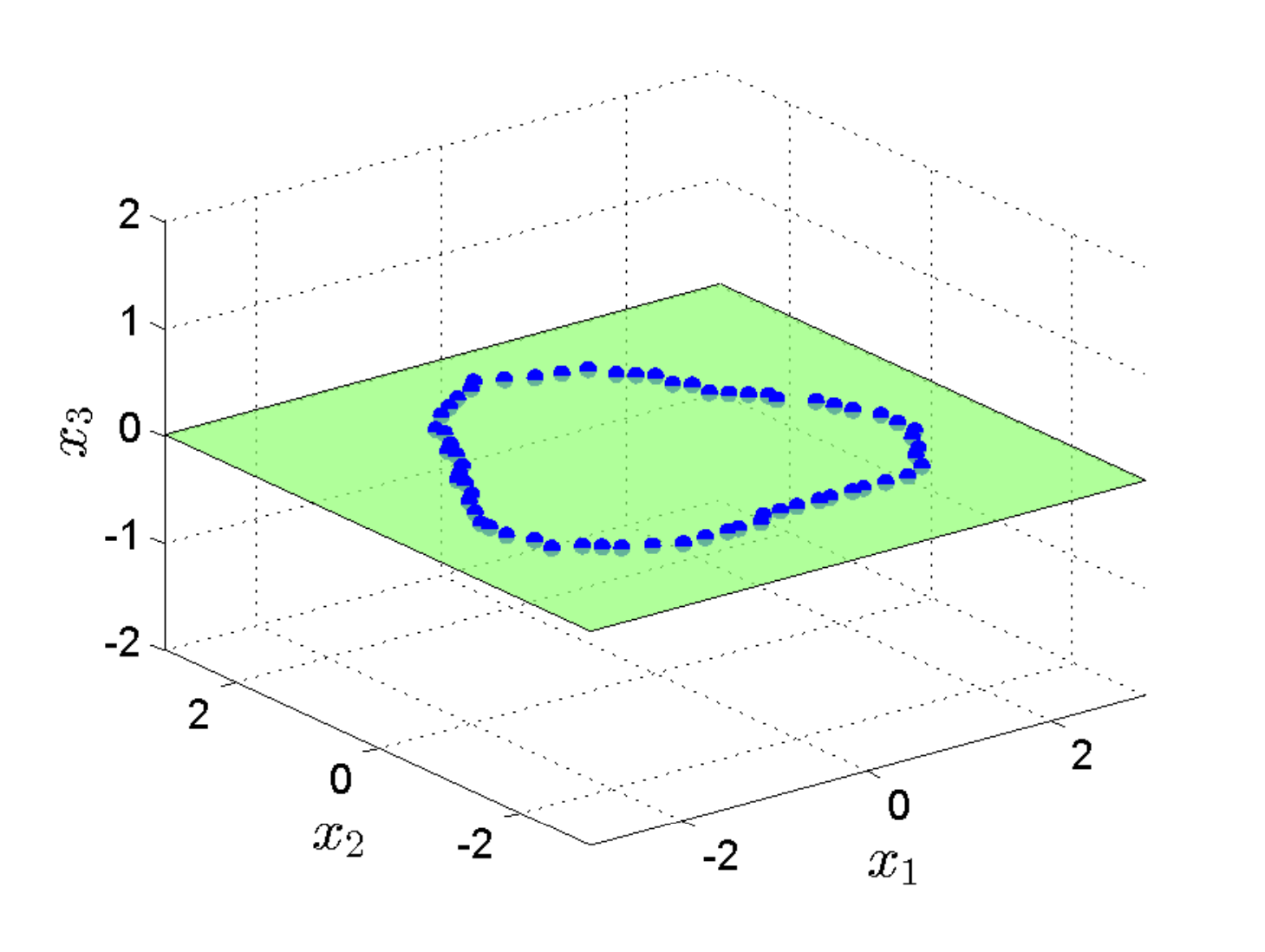}
}
\subfigure[]{
\includegraphics[scale=0.31]{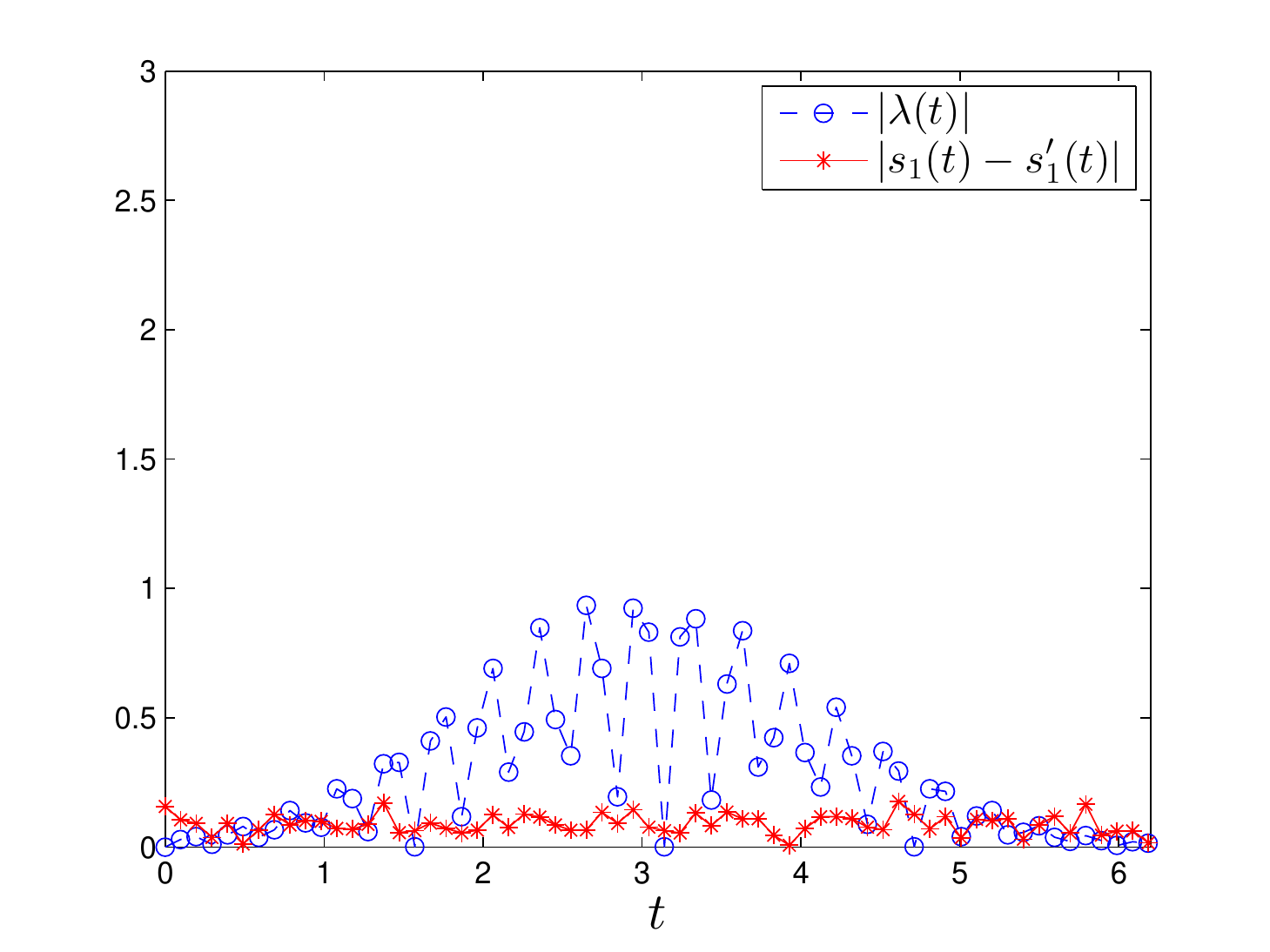}
}
\subfigure[]{
\includegraphics[scale=0.31]{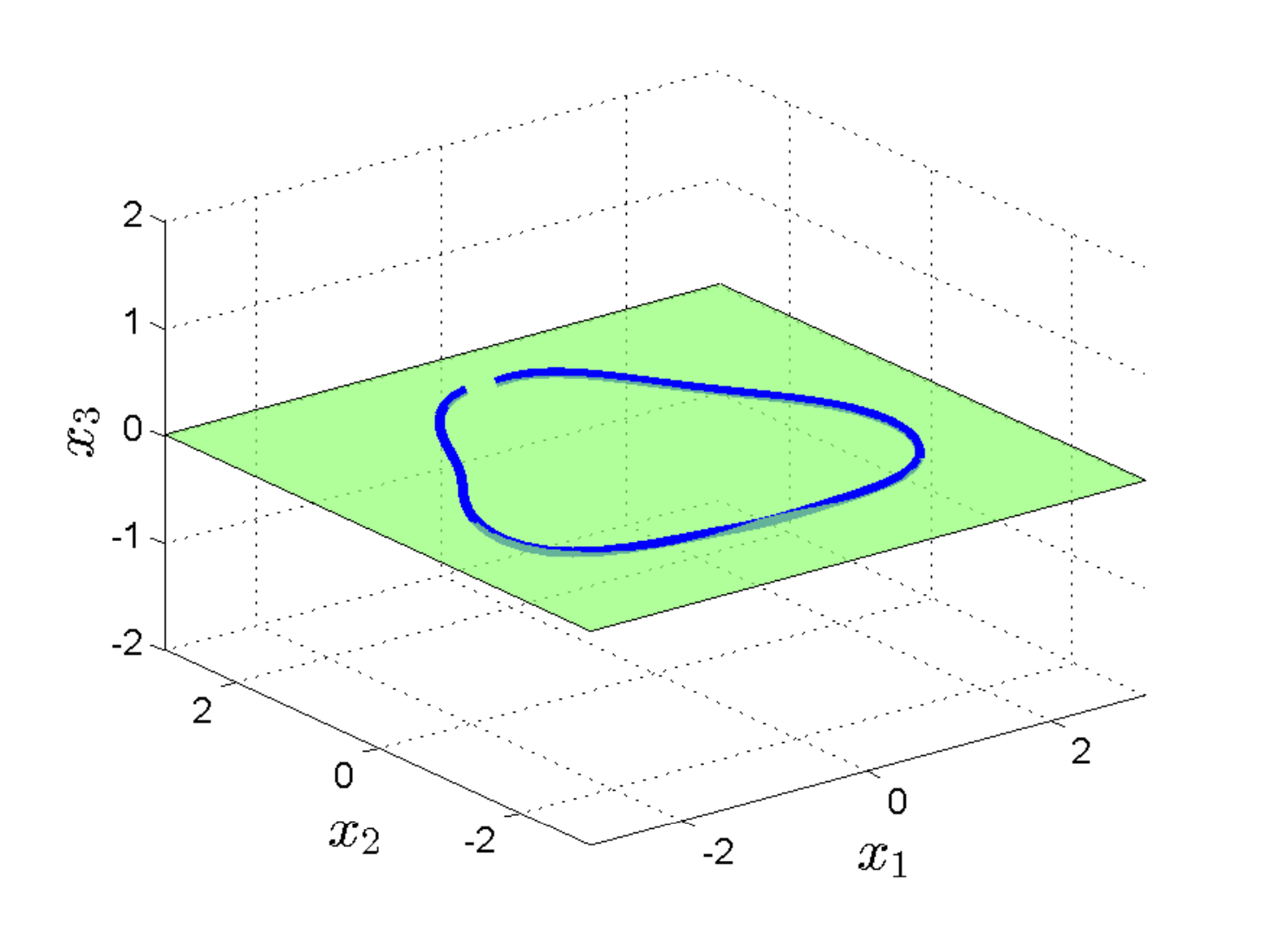}
}
\caption{Reconstruction of $s_1(t)$ in $\mathbb{R}^3$ (Example 3, Case 1). (a) Trajectory of the source point. (b) The reconstruction, $\epsilon=5\%$. (c) The error of the reconstruction. (d) The modified reconstruction, $\epsilon=5\%$. (e) The error of the modified reconstruction. (f) The smooth reconstruction with the Fourier expansion of order 5.} \label{fig006}
\end{figure}

\begin{figure}
\centering
\subfigure[]{
\includegraphics[scale=0.31]{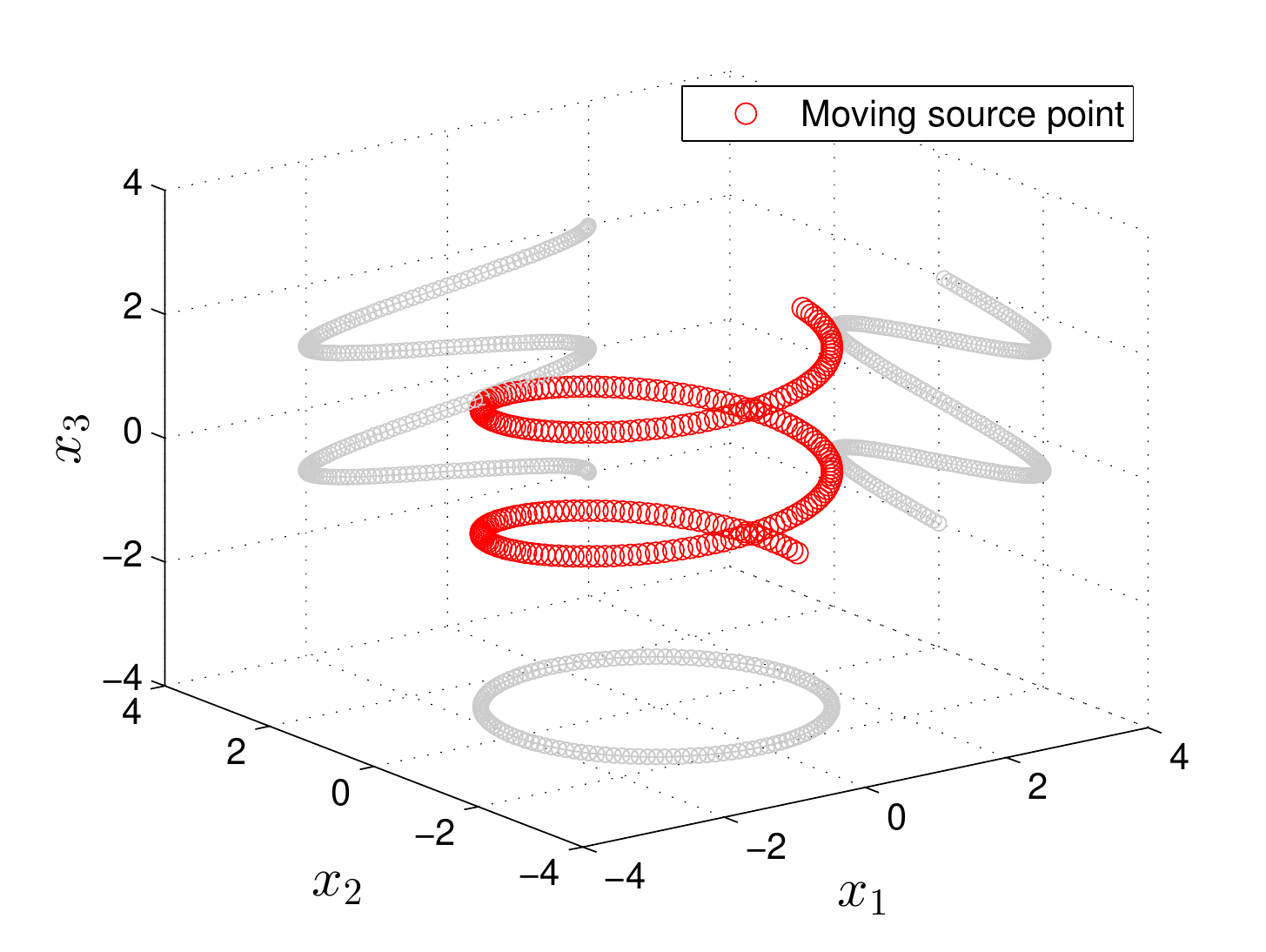}
}
\subfigure[]{
\includegraphics[scale=0.31]{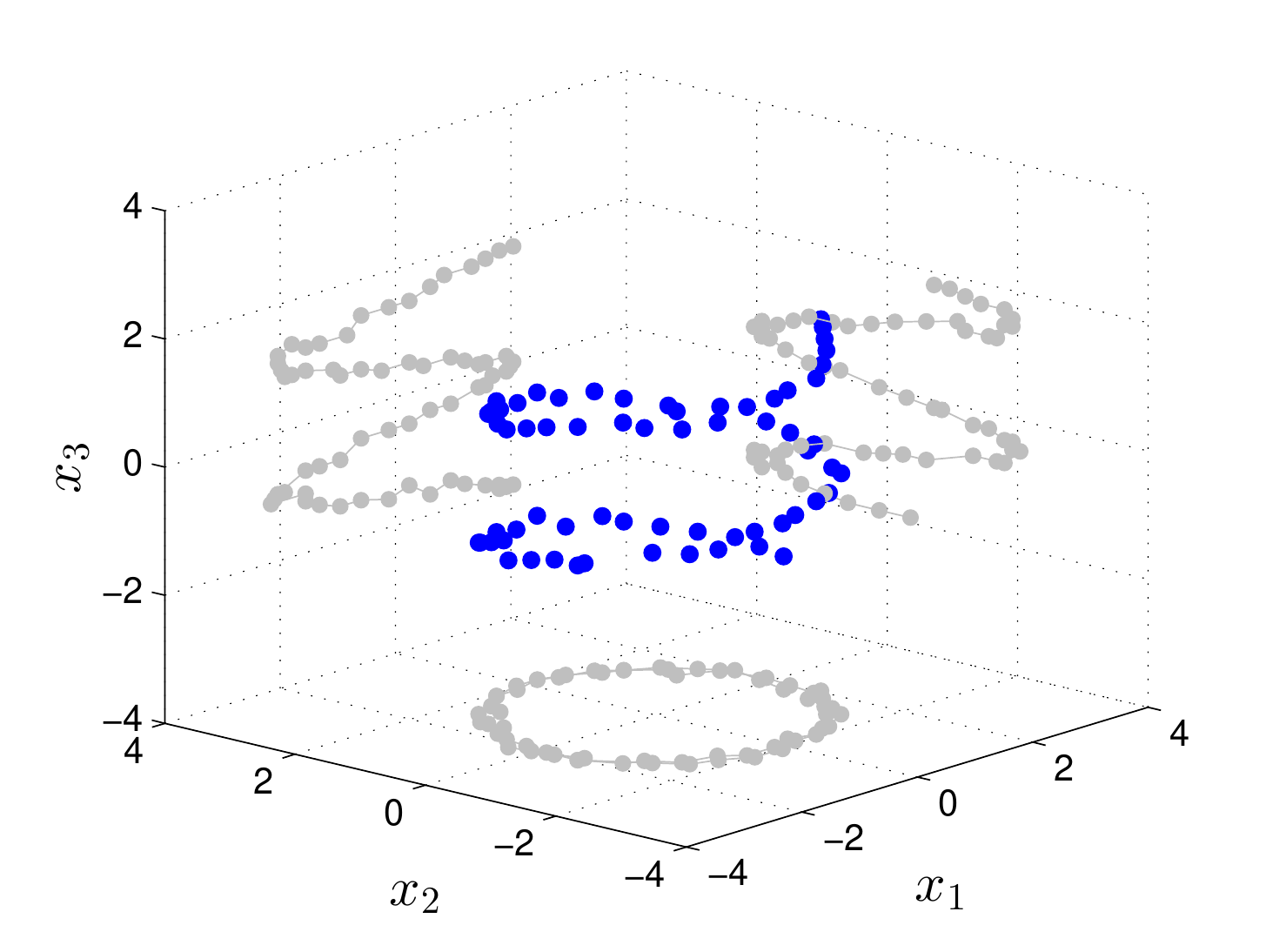}
}
\subfigure[]{
\includegraphics[scale=0.31]{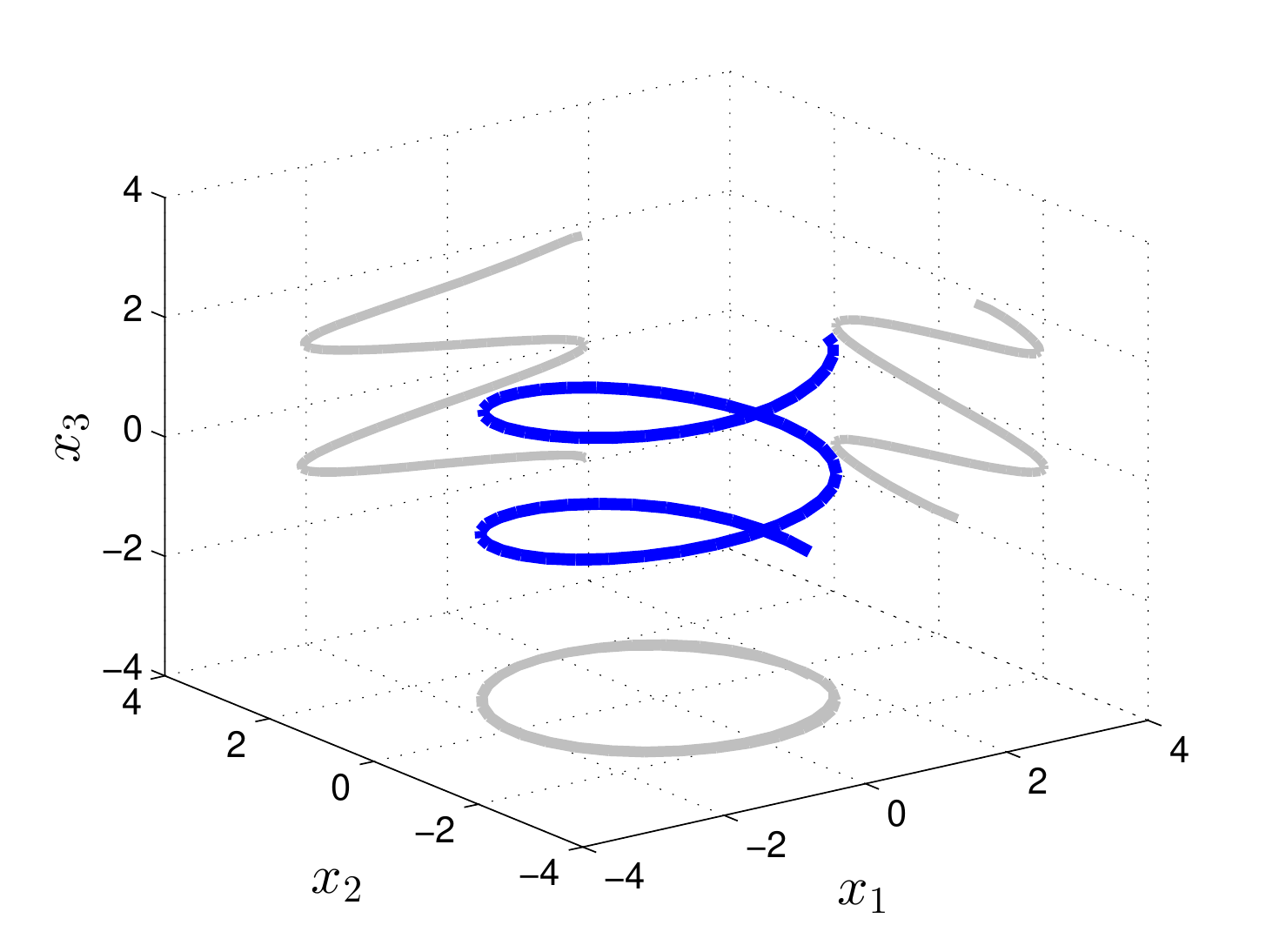}
}
\caption{Reconstruction of $s_2(t)$ in $\mathbb{R}^3$ (Example 3, Case 2). (a) Trajectory of the source point. (b) The reconstruction, $\epsilon=5\%$. (c) The smooth reconstruction given by the Fourier expansion of order 5. } \label{fig007}
\end{figure}

\noindent\textbf{Example 4.} As an addition of Example 3, the reconstruction of a handwritten Chinese character ``ai'' is considered. We choose $N_T=128$ in this experiment. The reconstructions can be seen in Figure \ref{fig008}.

The smooth reconstruction in this example is also provided by the Fourier expansion. However, the Chinese character ``ai'' has 5 strokes and can not be reconstruct with a single smooth curve. Thus the smooth reconstruction is provided respectively for each stroke. Since the point source moves faster in the gap between two strokes, we use the following strategy to provide the smooth reconstruction:\vspace{2mm}\\
{\small
(1) If $\max\left\{|s'(t_{k-1})-s'(t_{k})|,|s'(t_{k+1})-s'(t_{k})|\right\}>0.3$ for any $k=2,\ldots,N_T-1$, classify $s'(t_{k})$ as an end point of a stroke, or a point between two strokes. \\
(2) Separate the strokes of the character, and provide the smooth reconstruction of each stroke using the Fourier expansion.
}
\vspace{2mm}

As is shown in Figure \ref{fig008}, the algorithm is feasible to reconstruct the character with noise level $\epsilon=5\%$. The smooth reconstructions by the Fourier expansion with order 5 and order 3 are shown in Figure \ref{fig008}(c) and Figure \ref{fig008}(d), respectively.

\begin{Remark}
The smooth reconstruction by the Fourier expansion of order 5 indeed shows more details of the reconstruction than that of order 3. However, some of the details are caused by the noises. As is shown in Figure \ref{fig008}(c-d), the smooth reconstruction by the Fourier expansion of order 3 is better in this example.
\end{Remark}

\begin{figure}
\centering
\subfigure[]{
\includegraphics[scale=0.33]{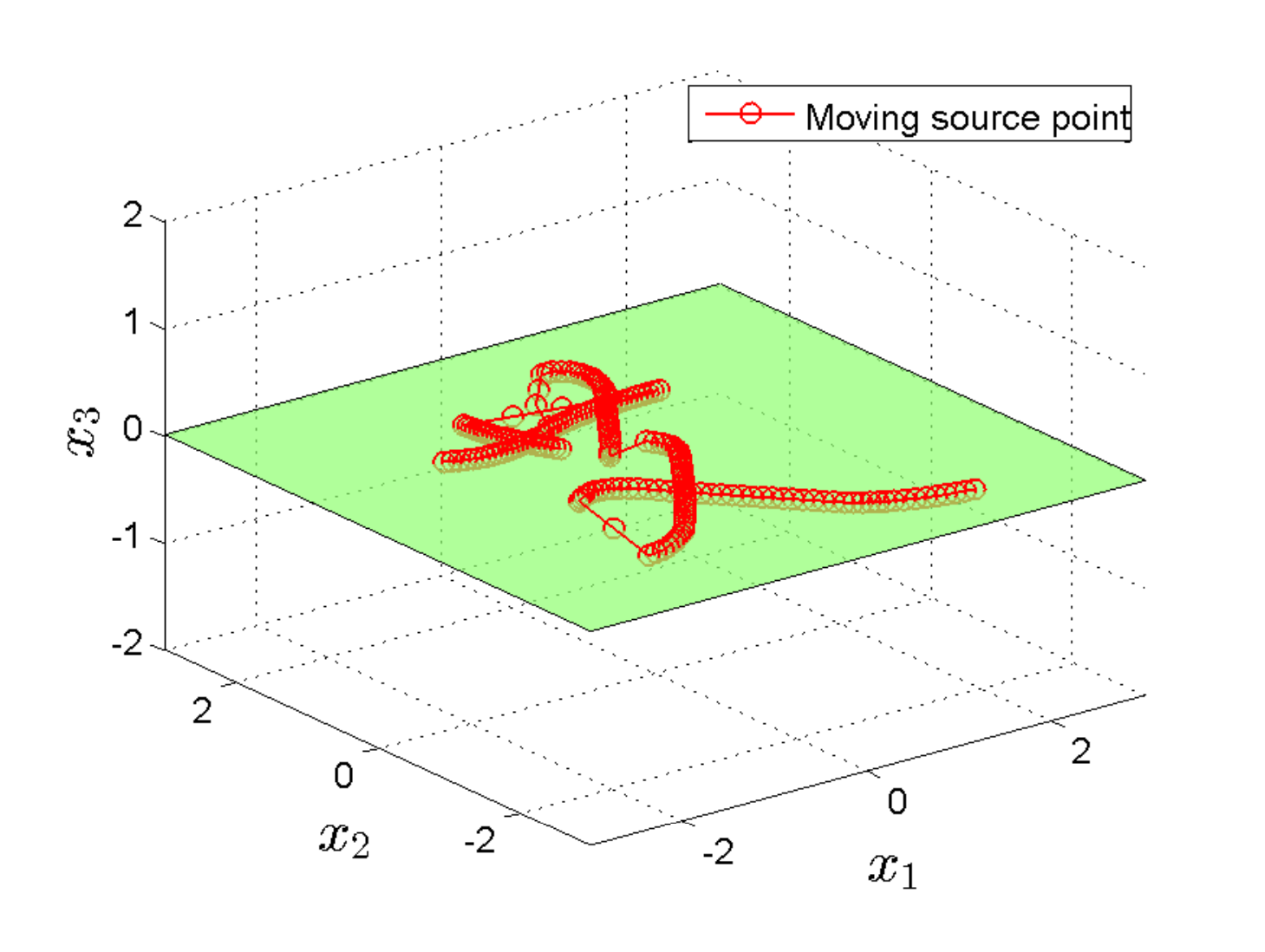}
}
\subfigure[]{
\includegraphics[scale=0.33]{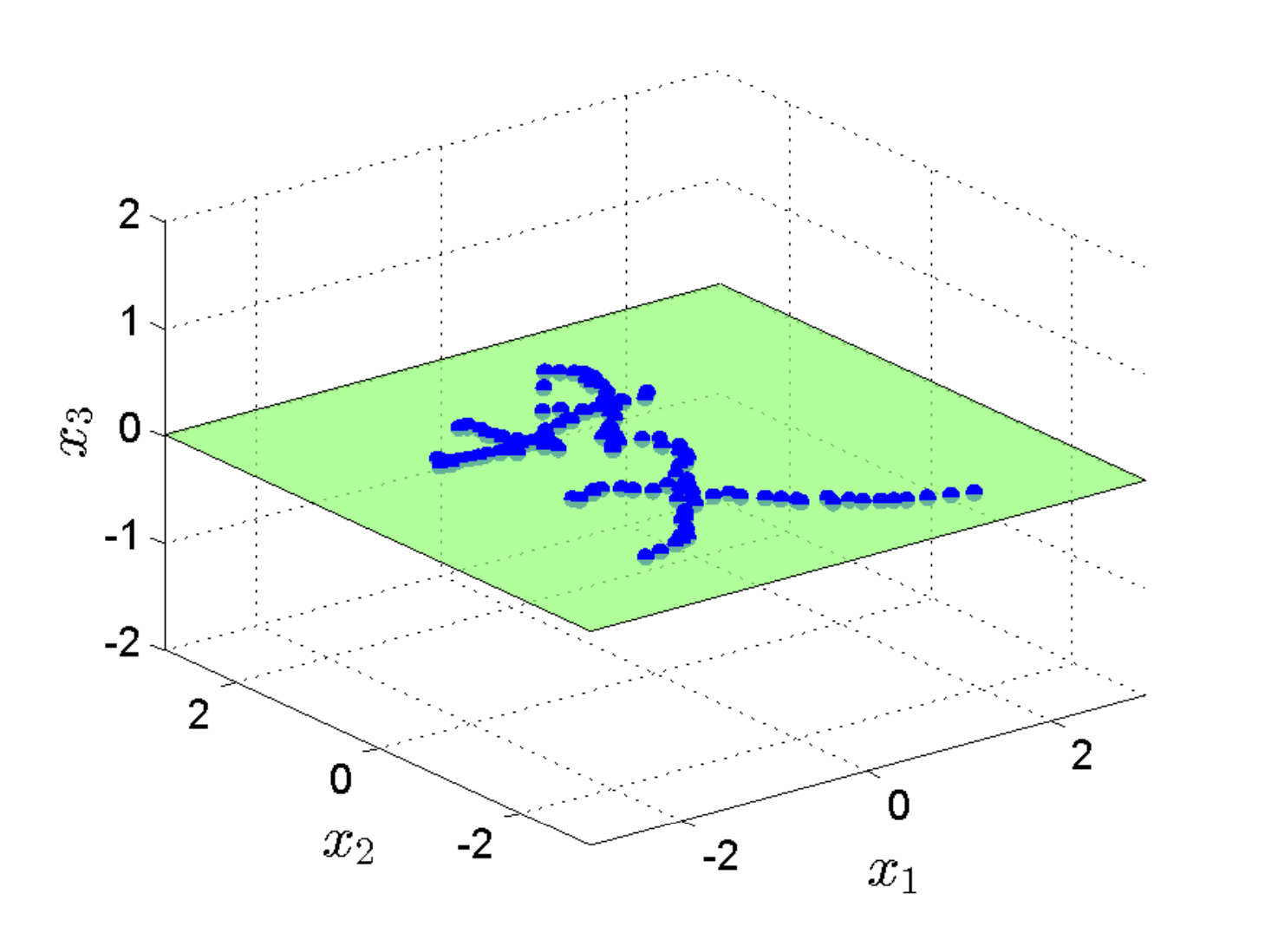}
}

\subfigure[]{
\includegraphics[scale=0.33]{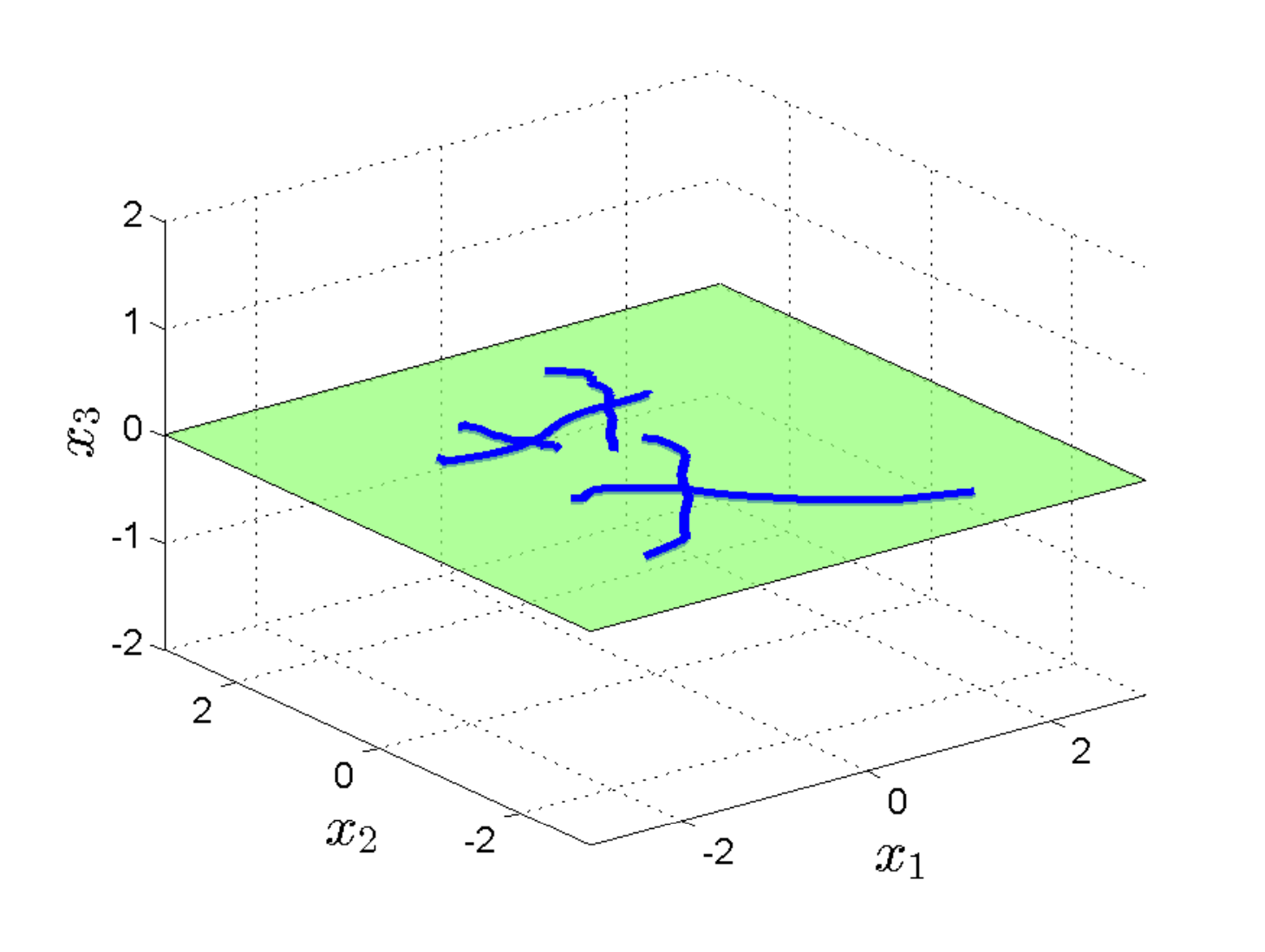}
}
\subfigure[]{
\includegraphics[scale=0.33]{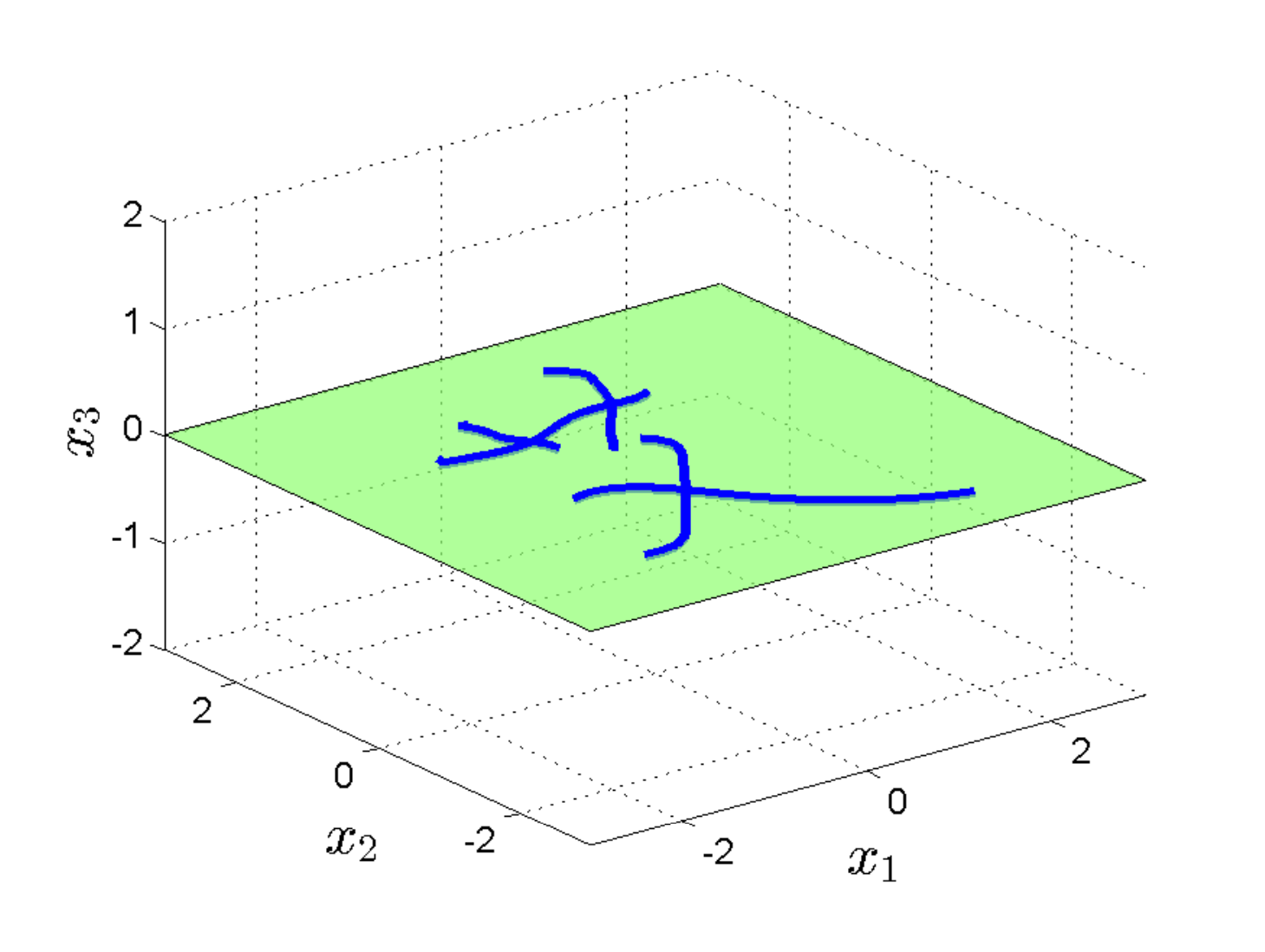}
}
\caption{Reconstruction of a handwritten Chinese character ``ai'' (Example 4). (a) Trajectory of the point source. (b) The reconstruction, $\epsilon=5\%$. (c) The smooth reconstruction with the Fourier expansion of order 5. (d) The smooth reconstruction with the Fourier expansion of order 3. } \label{fig008}
\end{figure}

\noindent\textbf{Example 5.} In this example, we are concerned about the reconstruction of the trajectory $s_3(t)=2\left(\frac{t}{\pi}\sin {2t}, \frac{t}{\pi}\cos {2t}, \frac{t}{\pi}-1\right)$ using 4 sensors. The sensing points are chosen as $(3,3,-3)$, $(3,-3,-3)$, $(-3,3,-3)$ and $(-3,-3,-3)$. The sampling points and the time discretization are chosen the same as that in Example 3. The reconstructions can be seen in Figure \ref{fig009}.

The error of the reconstruction with only 4 sensors is bigger than that of Example 3. Nevertheless, the smooth reconstruction ignores most of the error and the algorithm still works well.

\begin{figure}
\centering
\subfigure[]{
\includegraphics[scale=0.31]{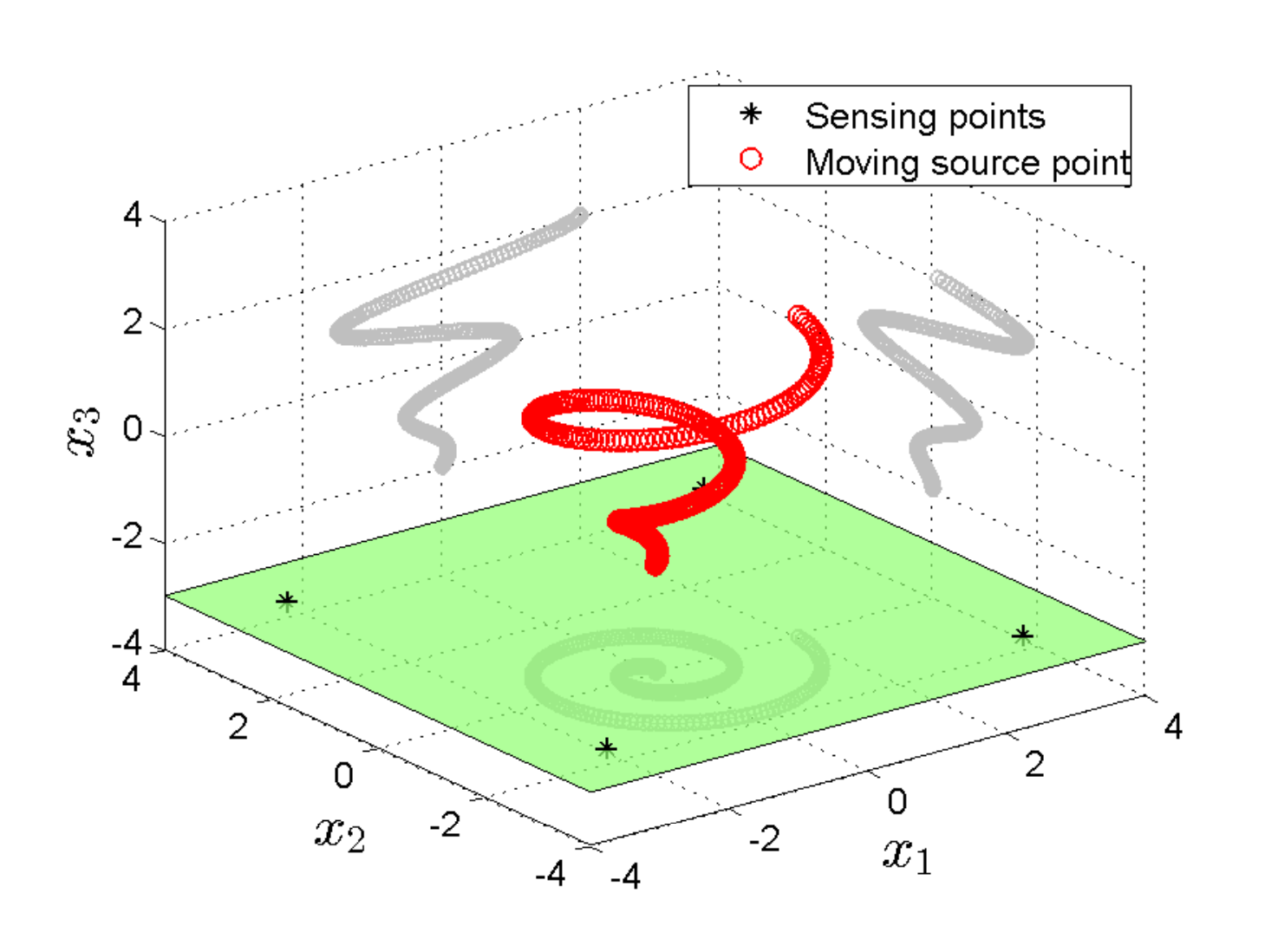}
}
\subfigure[]{
\includegraphics[scale=0.31]{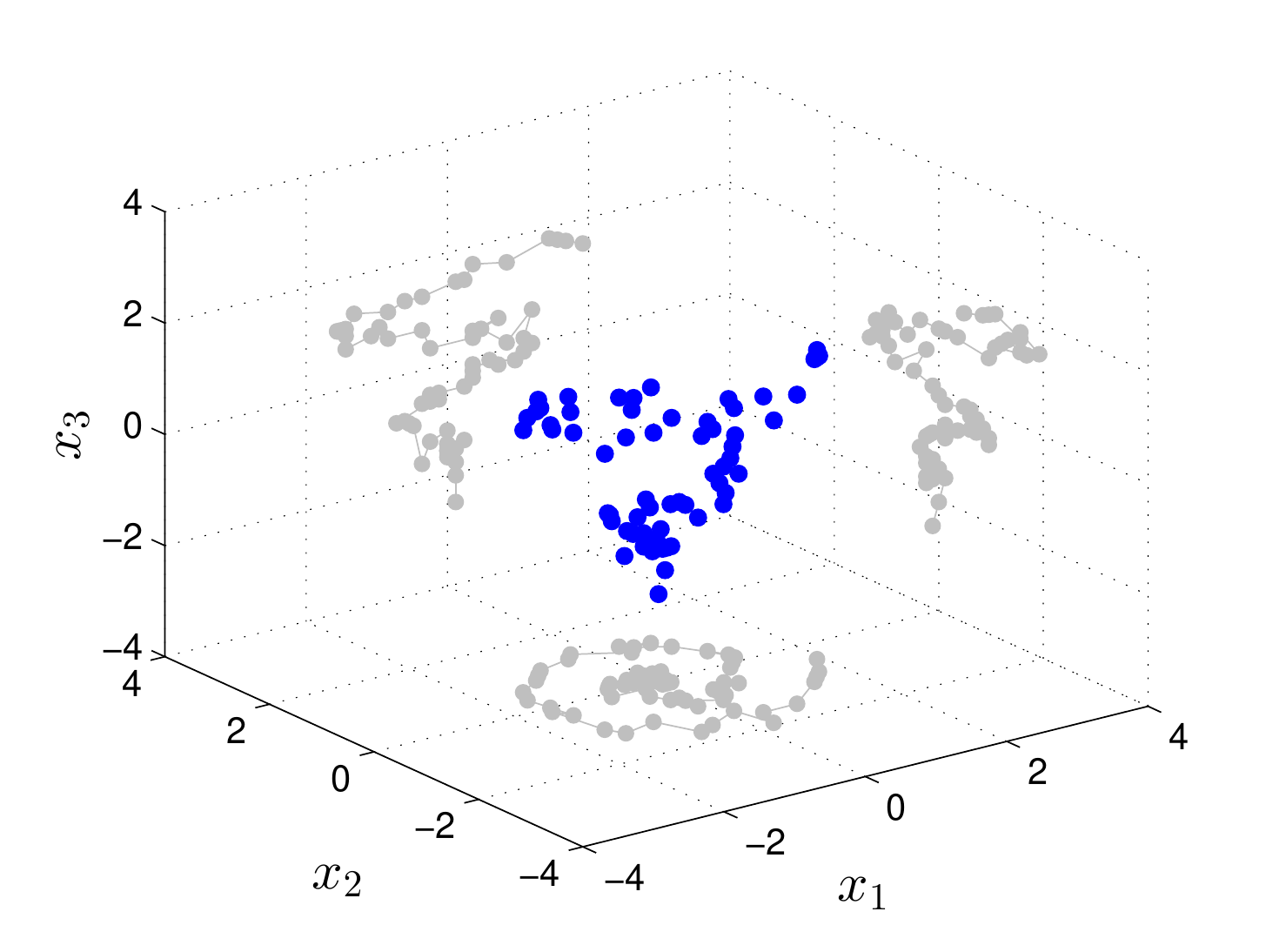}
}
\subfigure[]{
\includegraphics[scale=0.31]{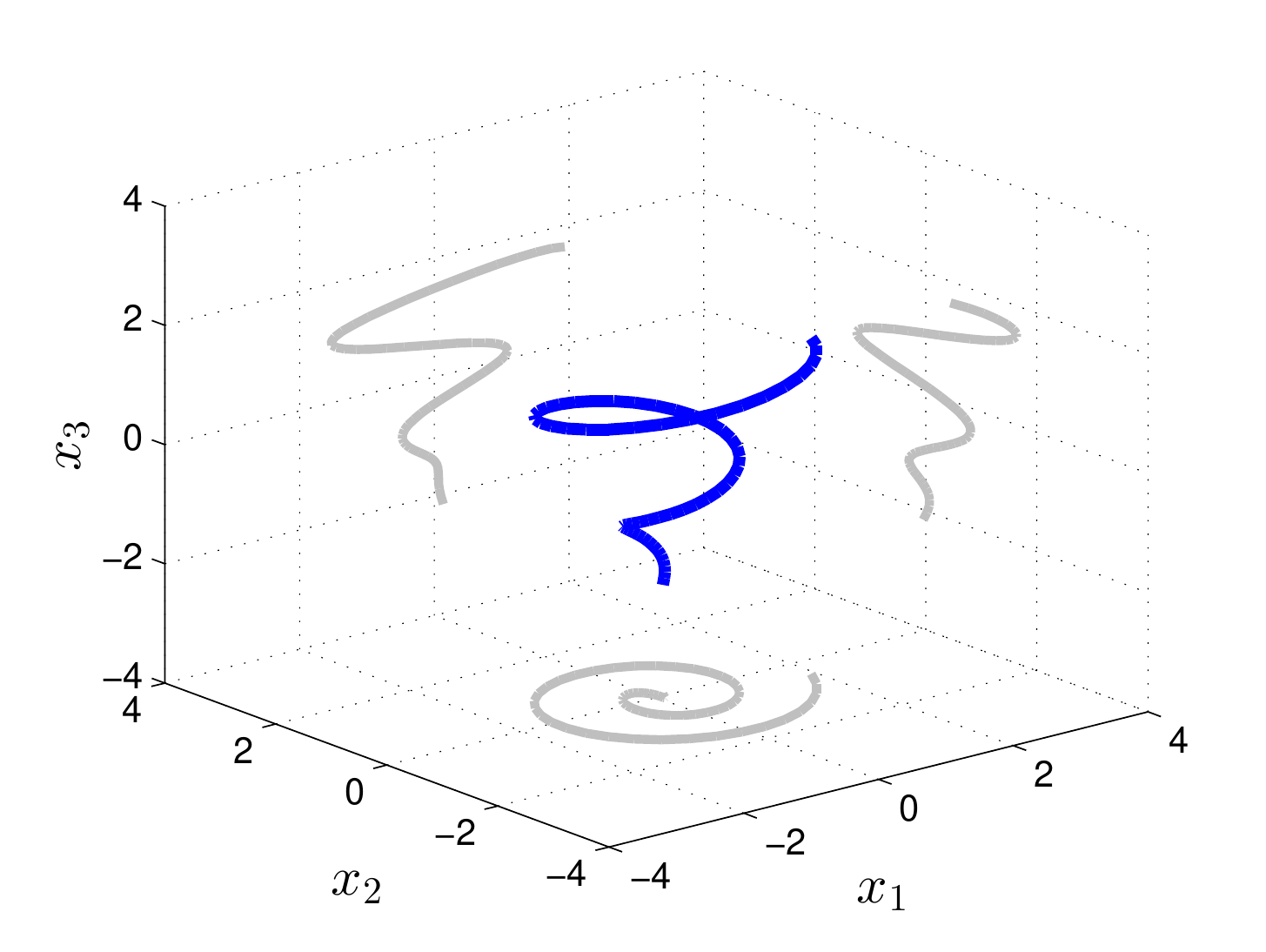}
}
\caption{Reconstruction of $s_3(t)$ with 4 sensors (Example 5). (a) Sketch of the example. (b) The reconstruction, $\epsilon=5\%$. (c) The smooth reconstruction with the Fourier expansion of order 3.} \label{fig009}
\end{figure}

\section{Conclusion}

We have considered the numerical simulation of the time dependent inverse source problems of acoustic waves. Modified methods of fundamental solutions have been established to reconstruct both multiple stationary sources and a moving point source. Moreover, the second modified method of fundamental solutions to reconstruct a moving point source has been modified to a simple sampling method. Several numerical examples have been provided to show the effectiveness of the proposed methods.

\section*{Acknowledgements}

The work of Bo Chen was supported by the NSFC (No. 11671170) and the Fundamental Research Funds for the Central Universities (Special Project for Civil Aviation University of China, No. 3122018L009). The work of Yukun Guo was supported by the NSFC (No. 11601107, 41474102 and 11671111). The work of Yao Sun was supported by the NSFC (No. 11501566).






\begin{thebibliography}{99}

\small
\bibitem{Ahmadabadi2009The}M. N. Ahmadabadi, M. Arab, and F. M. M. Ghaini. The method of fundamental solutions for the inverse space-dependent heat source problem. Engineering Analysis with Boundary Elements, 33(10):1231--1235, 2009.

\bibitem{Alves2009Iterative}C. Alves, R. Kress, and P. Serranho. Iterative and range test methods for an inverse source problem for acoustic waves. Inverse Problems, 25(5):055005, 2009.

\bibitem{Badia2011An}A. E. Badia and T. Nara. An inverse source problem for Helmholtz's equation from the Cauchy data with a single wave number. Inverse Problems, 27(10):105001, 2011.

\bibitem{Bao2014An}G. Bao, S. N. Chow, P. Li, and H. Zhou. An inverse random source problem for the Helmholtz equation. Mathematics of Computation, 83(285):215--233, 2014.

\bibitem{Bao2017Inverse}G. Bao, G. Hu, Y. Kian, and T. Yin. Inverse source problems in elastodynamics. Inverse Problems, 34(4):045009, 2017.

\bibitem{Bao2010A}G. Bao, J. Lin, and F. Triki. A multi-frequency inverse source problem. Journal of Differential Equations, 249(12):3443--3465, 2010.

\bibitem{Bao2015aresursive}G. Bao, S. Lu, W. Rundell, and B. Xu. A recursive algorithm for multi-frequency acoustic inverse source problems. SIAM Journal on Numerical Analysis, 53(3):1608--1628, 2015.

\bibitem{ChenB2016time}B. Chen, F. Ma, and Y. Guo. Time domain scattering and inverse scattering problems in a locally perturbed half-plane. Applicable Analysis, 96(8):1303--1325, 2017.

\bibitem{ChenB2019Method}B. Chen, Y. Sun, and Z. Zhuang. Method of fundamental solutions for a Cauchy problem of the Laplace equation in a half-plane. Boundary Value Problems, 2019(34):1--14, 2019.

\bibitem{Cheng2002UNIQUENESS}J. Cheng, G. Ding, and M. Yamamoto. Uniqueness along a line for an inverse wave source problem. Communications in Partial Differential Equations, 27(9--10):2055--2069, 2002.

\bibitem{Cheng2005The}J. Cheng, L. Peng, and M. Yamamoto. The conditional stability in line unique continuation for a wave equation and an inverse wave source problem. Inverse Problems, 21(6):1993--2007, 2005.


\bibitem{De2015An}M. V. De Hoop and J. Tittelfitz. An inverse source problem for a variable speed wave equation with discrete-in-time sources. Inverse Problems, 31(7):075007, 2015.

\bibitem{Fournier2017Matched}J. Fournier, J. Garnier, G. Papanicolaou, and C. Tsogka. Matched-filter and correlation-based imaging for fast moving objects using a sparse network of receivers. Siam Journal on Imaging Sciences, 10(4):2165--2216, 2017.

\bibitem{Garnier2015Super}J. Garnier and M. Fink. Super-resolution in time-reversal focusing on a moving source. Wave Motion, 53:80--93, 2015.

\bibitem{Guo2016A}Y. Guo, D. H\"{o}mberg, G. Hu, J. Li, and H. Liu. A time domain sampling method for inverse acoustic scattering problems. Journal of Computational Physics, 314:647--660, 2016.

\bibitem{guo2013toward}Y. Guo, P. Monk, and D. Colton. Toward a time domain approach to the linear sampling method. Inverse Problems, 29(9):095016, 2013.

\bibitem{Isakov1998Inverse}V. Isakov. Inverse Problems for Partial Differential Equations. Springer-Verlag, New York, 1998.

\bibitem{Lassas2010Inverse}M. Lassas and L. Oksanen. Inverse problem for wave equation with sources and observations on
disjoint sets. Inverse Problems, 26(8):085012, 2010.

\bibitem{Li2013Two}J. Li, H. Liu, Z. Shang, and H. Sun. Two single-shot methods for locating multiple electromagnetic scatterers. SIAM Journal on Applied Mathematics, 73(4):1721--1746, 2013.

\bibitem{Lijingzhi2018On}J. Li, H. Liu, and H. Sun. On a gesture-computing technique using electromagnetic waves. Inverse Problems and Imaging, 12(3):677--696, 2018.

\bibitem{Li2008Multilevel}J. Li, H. Liu, and J. Zou. Multilevel linear sampling method for inverse scattering problems. Journal of Computational Physics, 30(3):1228--1250, 2008.

\bibitem{Li2009Strengthened}J. Li, H. Liu, and J. Zou. Strengthened linear sampling method with a reference ball. SIAM Journal on Scientific Computing, 31(6):4013--4040, 2009.

\bibitem{Li2014Locating}J. Li, H. Liu, and J. Zou. Locating multiple multiscale acoustic scatterers. SIAM Multiscale Modeling and Simulations, 12(3):927--952, 2014.

\bibitem{Li2011An}P. Li. An inverse random source scattering problem in inhomogeneous media. Inverse Problems, 27(3):035004, 2011.

\bibitem{Li2016Increasing}P. Li and G. Yuan. Increasing stability for the inverse source scattering problem with multi-frequencies. Inverse Problems and Imaging, 11(4):745--759, 2016.

\bibitem{Lim1994On}P. H. Lim and J. M. Ozard. On the underwater acoustic field of a moving point source. i. range-independent environment. Journal of the Acoustical Society of America, 95(1):131--137, 1994.

\bibitem{Lim1994On2}P. H. Lim and J. M. Ozard. On the underwater acoustic field of a moving point source. ii. range-dependent environment. Journal of the Acoustical Society of America, 95(1):138--151, 1994.

\bibitem{Matsumoto2003A}M. Matsumoto, M. Tohyama, and H. Yanagawa. A method of interpolating binaural impulse responses for moving sound images. Acoustical Science and Technology, 24(5):284--292, 2003.

\bibitem{Nakaguchi2012An}E. Nakaguchi, H. Inui, and K. Ohnaka. An algebraic reconstruction of a moving point source for a scalar wave equation. Inverse Problems, 28(6):065018, 2012.

\bibitem{sayas2011retarded}F. J. Sayas. Retarded Potentials and Time Domain Boundary Integral Equations: a Road-map. Springer Series in Computational Mathematics, Switzerland, 2016.


\bibitem{Sun2014Modified}Y. Sun. Modified method of fundamental solutions for the Cauchy problem connected with the Laplace equation. International Journal of Computer Mathematics, 91(10):2185--2198, 2014.

\bibitem{Sun2017Indirect}Y. Sun. Indirect boundary integral equation method for the Cauchy problem of the Laplace equation. Journal of Scientific Computing, 71(2):469--498, 2017.

\bibitem{Ton2003An}B. A. Ton. An inverse source problem for the wave equation. Nonlinear Analysis, 55(3):269--284, 2003.

\bibitem{wang2017mathematical}X. Wang, Y. Guo, J. Li, and H. Liu. Mathematical design of a novel input/instruction device using a moving emitter. Inverse Problems, 33(10):105009, 2017.

\bibitem{Wei2010Convergence}T. Wei and D. Y. Zhou. Convergence analysis for the Cauchy problem of Laplace's equation by a regularized method of fundamental solutions. Advances in Computational Mathematics, 33(4):491--510, 2010.

\bibitem{Zhang2015Fourier}D. Zhang and Y. Guo. Fourier method for solving the multi-frequency inverse source problem for the Helmholtz equation. Inverse Problems, 31(3):035007, 2015.

\bibitem{Zhang2018Uniqueness}D. Zhang and Y. Guo. Uniqueness results on phaseless inverse scattering with a reference ball. Inverse Problems, 34(8):085002, 2018.

\bibitem{Zhang2018Retrieval}D. Zhang, Y. Guo, J. Li, and H. Liu. Retrieval of acoustic sources from multi-frequency phaseless data. Inverse Problems, 34(9):094001, 2018.

\bibitem{Zhang2018Locating}D. Zhang, Y. Guo, J. Li, and H. Liu. Locating multiple multipolar acoustic sources using the direct sampling method. Communications in Computational Physics, 25(5):1328--1356, 2019.

\end{thebibliography}


\end{document}